\documentclass[12pt, leqno]{article}

\usepackage{graphicx}
\usepackage{color}

\usepackage{tabularx}

\usepackage{amsmath}
\usepackage{amssymb}
\usepackage{amsthm}

\usepackage{enumitem}

\newtheorem{thm}{Theorem}
\newtheorem{lem}{Lemma}

\newtheorem{rem}{Remark}

\theoremstyle{definition}
\newtheorem{defi}{Definition}

\newtheorem{nota}{Notation}

\usepackage{hyperref}

\newcommand{\RR}{\mathbb{R}}

\newcommand{\NN}{\mathbb{N}}
\newcommand{\norm}[2]{||#1||_{#2}}
\newcommand{\skp}[2]{\left\langle #1 \right\rangle_{#2}}
\newcommand{\RE}{\textnormal{Re}\,}

\newcommand{\diag}{\textnormal{diag}}

\title{Nonlocalized damping estimates for hyperbolic systems of balance laws in one space dimension}
\author{Johannes B{\"a}rlin}

\begin{document}
\maketitle
\begin{abstract}
    In this paper, we present a new approach to obtain so-called
    damping estimates for self-similar solutions to general hyperbolic relaxation systems. Such damping estimates are an important part
    of the stability theory of relaxation profiles with and without subshock where they enable the closure
    of nonlinear stability arguments \cite{MZ05,YZ20}. We extend the damping estimates
    obtained in \cite{MZ05,YZ20} from the $L^2$-case to the $L^\infty$-case and, at
    the same time, generalize the $L^2$-estimates to the general non-symmetric
    setting. Compared to previous deductions of damping estimates, the method chosen here is
    considerably simpler and avoids delicate Kawashima-type estimates \cite{MZ05,YZ20} and the use of pseudo-differential calculus \cite{Zumbrun_26}. The novel nonlocalized damping estimates open the door to a general stability theory of shock profiles
    of hyperbolic relaxation systems under nonlocalized perturbations in one space dimension.
\end{abstract}
\section{Introduction}
\label{sec:Introduction}
Let $N \in \NN$, and let $\mathcal{U} \subset \RR^N$ be a domain. Consider a hyperbolic system of balance laws (cf.~\cite{D16})
\begin{align}
    \label{equ:general_balance_law}
    U_t + F(U)_x = q(U)
\end{align}
with $(t,x) \in [0,\infty) \times \RR, U \in \mathcal{U}$, smooth flux
$F: \mathcal{U} \rightarrow \RR^{N}$, and smooth source term $q: \mathcal{U} \rightarrow
\RR^N$. For $U \in \mathcal{U}$, denote the differentials of $q$ and $F$ by
\begin{align*}
    Q(U) := dq(U) \textnormal{ and } A(U) := dF(U).
\end{align*}
The system \eqref{equ:general_balance_law} is assumed to be constantly hyperbolic \cite{BGD07} which means that 
\begin{quote}
$A(U)$ has $J \in \NN$ distinct real and semi-simple eigenvalues
\[\lambda_1(U) < \lambda_2(U) < \ldots < \lambda_J(U)\]
with constant multiplicities $m_1, \ldots, m_J \in \NN$ and $\sum_{j=1}^J m_j = N$.
\end{quote}

We are interested in damping estimates for certain time-independent solutions $\bar{U}=\bar{U}(x)$ of \eqref{equ:general_balance_law} that will be referred
to as  profiles. Such profiles are often obtained from traveling-wave solutions of (a related equation of) \eqref{equ:general_balance_law} when moving into the wave's rest frame. A damping estimate is an important part of their nonlinear stability analysis as stationary solutions of \eqref{equ:general_balance_law}.

Let $\dot{\RR} := \RR \backslash \lbrace 0 \rbrace$ be the punctured real line. In this paper two different types of profiles will be considered:
\begin{defi}
    \label{def:smooth_profile}
    A smooth function $\bar{U}: \RR \rightarrow \RR^N$ is called a smooth profile of \eqref{equ:general_balance_law} if $U(t,x)= \bar{U}(x)$ is a stationary solution of
    \eqref{equ:general_balance_law}, i.e.~$\bar{U}$ solves
    \begin{align}
        \label{equ:smooth_profile}
        A(U)U_x = q(U) \textnormal{ on } \RR,
    \end{align}
    while approaching end states $U^-, U^+ \in \RR^N$ at exponential rate,
    \begin{align}
        \label{def:exponential_decay_smooth}
        \left|\partial_x^k \left(\bar{U} - U^\pm\right)(x)\right| \leq c_k
        e^{-\theta_k |x|}, \quad x \gtrless 0, \quad k =0, 1, \ldots.
    \end{align}
\end{defi}
\begin{defi}
    \label{def:discontinuous_profile}
    A smooth function $\bar{U}: \dot{\RR} \rightarrow \RR^N$ is called a discontinuous profile, or profile with subshock, of \eqref{equ:general_balance_law} if $\bar{U}$ solves
    \begin{align}
        \label{equ:discontinuous_profile}
        F(U)_x = q(U) \textnormal{ on } \dot{\RR},
    \end{align}
    with well-defined one-sided limits at $x=0$
    \begin{align*}
        \lim\limits_{x \rightarrow 0^-} U(x) = U(0^-) \neq U(0^+) = \lim\limits_{x \rightarrow 0^+} U(x),
    \end{align*}
    that satisfy the Rankine-Hugoniot condition
    \begin{align}
        \label{def:subshock}
        0 = [F(\bar{U})]_0 := F(U(0^+)) - F(U(0^-)),
    \end{align}
    while approaching end states $U^-, U^+ \in \RR^N$ at exponential rate,
    \begin{align}
        \label{def:exponential_decay_dis}
        \left|\partial_x^k \left(\bar{U} - U^\pm\right)(x)\right| \leq c_k
        e^{-\theta_k |x|}, \quad x \gtrless 0, \quad k =0, 1, \ldots.
    \end{align}
    The jump in $\bar{U}$ at $x=0$ is called a subshock.
\end{defi}
\begin{nota}
    For $x_0 \in \RR$ and a function $h=h(x)$ defined on both sides of $x_0$ with one-sided limits $\lim_{x\rightarrow x^-_0}h(x) = h(x_0^-)$ and $\lim_{x\rightarrow x^+_0}h(x)
    = h(x_0^+)$ we set
    \[[h]_{x_0} := h(x_0^+) - h(x_0^-).\]
\end{nota}
\begin{rem}
    Note that $q$ necessarily has to vanish at the end states $U^\pm$ of a smooth or discontinuous profile $\bar{U}$,
    \begin{align*}
        q(U^\pm) = 0.
    \end{align*}
\end{rem}
\begin{rem}
    The presented results concerning smooth profiles hold for general nonlinear hyperbolic systems with source term of the form
    \begin{align*}
        U_t + A(U)U_x = q(U).
    \end{align*}
    The assumption that \eqref{equ:general_balance_law} consists of balance laws is needed in order to make sense of the subshock condition \eqref{def:subshock} in the
    discontinuous profile case.
\end{rem}
The term \textit{profile} is used for both smooth and discontinuous profiles if at some point the distinction is not important or the context clearly reveals what type of
profile is meant.

Let $\bar{U}$ be a profile of \eqref{equ:general_balance_law}. Define the shifted perturbation variable $U$ by
\[U(t,x) = \tilde{U}(t,x+\eta(t)) - \bar{U}(x)\]
where $\tilde{U} = \tilde{U}(t,x)$ is another solution of \eqref{equ:general_balance_law} and $\eta=\eta(t)$ is a smooth shift, both defined on some
time interval $[0,T]$. If $\bar{U}$ is smooth, then $\tilde{U}$ is assumed to be smooth, too, and $\eta$ is a shift that in applications is determined by means not discussed and not relevant
here (the reader is referred to the references in the discussion on the stability of relaxation/shock profiles below). If $\bar{U}$ is discontinuous, then $\tilde{U}$ is assumed to be smooth everywhere except at the interface
$\lbrace (t,x):\, t \in [0,T], x=\eta(t)\rbrace$ where a single subshock with jump condition
\[\dot{\eta}(t)[\tilde{U}(t)]_{\eta(t)} = [F(\tilde{U}(t))]_{\eta(t)}\] 
is located with $\dot{\eta} = d\eta /dt = \eta_t$.

A damping estimate for $U$ takes the general form
\begin{align}
    \label{ieq:general_damping_estimate}
    \norm{U(t)}{K}^p \leq C e^{-\theta t} \norm{U(0)}{K}^p + C \int\limits_0^t e^{-\theta(t-s)} \left(\norm{U(s)}{0}^p + |\dot{\eta}(s)|^p \right) ds
\end{align}
where $\theta, C > 0$ are (uniform) positive constants, $\norm{f}{K}$ is a fixed norm that takes into account 
$K \in \lbrace 0, 1, \ldots \rbrace$ derivatives of functions $f$ lying in some function space $\mathcal{F}(\Omega)$ for an open set
$\Omega \subset \RR$, and $p \in \NN$ is some fixed exponent adapted to the norm $\norm{\cdot}{K}$. 

What is an estimate like \eqref{ieq:general_damping_estimate} used for? In the course of a nonlinear stability argument for $\bar{U}$, one often obtains estimates on $U$ and
$\dot{\eta}(t)$ at the cost of losing control over a certain number of derivatives of $U$. A damping estimate like \eqref{ieq:general_damping_estimate} reestablishes control over these
derivatives of $U$ and allows for the closure of the nonlinear stability argument at hand.

More concretely, the context we have in mind is the one where the system of balance laws \eqref{equ:general_balance_law} comprises a hyperbolic relaxation system (cf.~\cite{Whitham_74, Liu_87,Yong_01,MZ02}) and the profiles
are relaxation profiles, which are certain heteroclinic connections of end states $U^-$ and $U^+$. Damping estimates like \eqref{ieq:general_damping_estimate} are used in the previously described way in the proofs of
nonlinear stability of relaxation profiles in the smooth case \cite{MZ02,MZ05} and in the discontinuous case
\cite{YZ20}. They are of particular importance for large-amplitude profiles where comparability with a single (stable) constant state fails \cite{MZ05}. Other contexts where
damping estimates play an important role include the stability analysis of, e.g., viscous shock profiles \cite{MZ04_2,MZ04_1}, viscous roll waves \cite{JZN_11}, and periodic solutions of viscous conservation laws \cite{JNRZ_14} (see \cite{BW24} for an
application in the context of detonation waves).

They are also found in the stability study of constant state solutions of systems like \eqref{equ:general_balance_law}, see e.g.~\cite{S25}. There, the origins of the argument go back to the thesis of Kawashima \cite{K83}
and beyond (see the brief discussion in \cite{Z_04_1}, and the more extensive one in \cite{HoZ95}), where dissipative structures in hyperbolic-parabolic systems are systematically studied. Global existence results as well as decay estimates
for perturbations of constant states are derived on the basis of then ``nonstandard'' energy estimates which propagate damping that at first sight is only experienced by a certain portion of the state variables onto the entire set of 
state variables.

For systems of size $N=2$, an existence criterion for relaxation profiles of arbitrary amplitude
can be found in the influential work of Liu \cite{Liu_87}. While large-amplitude profiles are known to exist for specific models, e.g., \cite{Broadwell_64,Aregba-Driolle_Hanouzet_11,Wang_13},
a systematic existence theory for smooth small-amplitude relaxation profiles in systems \eqref{equ:general_balance_law} of size $N \geq 2$ is well-established 
\cite{Yong_Zumbrun_00,MZ02,Dressel_Yong_06}.

In general, one does not expect the existence of arbitrarily large amplitude relaxation profiles. In fact, often a natural transition from smooth profiles to discontinuous profiles can be observed when
the amplitude $|U^+-U^-|$ is increased (see, e.~g., \cite{Boillat_Ruggeri_98}, and the discussion in \cite{YZ20}). For example, the inviscid Saint-Venant equations exhibit such a behavior: In their pioneering work on the stability of
hydraulic shock profiles, Yang and Zumbrun \cite{YZ20} show that for Froude numbers $F \in (0,2)$, and fixed $U^-$, a smooth profile only exists up to a certain threshold on the amplitude. 
Beyond the threshold a profile connecting $U^-$ and $U^+$, left to right, continues to exist but with a single discontinuity, or subshock, in between $U^-$ and $U^+$. A similar observation can be
made for the related $3 \times 3$ Richard-Gavrilyuk model \cite{Richard_Gavrilyuk_12,Richard_Gavrilyuk_13,Rodrigues_Yang_Zumbrun_25}.

The systematic use of damping estimates like \eqref{ieq:general_damping_estimate} began with the work by Mascia and Zumbrun \cite{MZ05} considering large-amplitude relaxation profiles and culminated in the
widely applicable pseudo-differential damping estimates derived in \cite{Zumbrun_26} for one- and multi-dimensional space. The results of the present paper lie, in a sense, between those two cited works. We do not aim for the same 
generality as in \cite{Zumbrun_26} but generalize the results of \cite{MZ05} to the non-symmetric setting, and, at the same time, simplify the derivation of localized, that is to say $L^2$-based, damping estimates considerably. Where 
delicate Kawashima-type estimates were carried out in \cite{MZ05,YZ20}, our argument instead rests on a two-step diagonalization procedure together with an energy method using Goodman-type weights \cite{G91} as in \cite{MZ05}. In
the discontinuous case, the observation that a Lopatinski condition at the subshock implies maximally dissipative boundary condition \cite{YZ20} is used, too.

The first step of the diagonalization procedure is a standard decomposition of the solution vector along eigenvector fields; the second step is a variable transform that combines higher and lower derivatives that eliminates
troublesome large off-diagonal terms much in the spirit of Shizuta and Kawashima \cite{SK85}. This happens directly on the nonlinear, non-Fourier transformed level instead of the linearized frequency level as in \cite{SK85}.

As a novelty and through an application of the method of characteristics,
the approach is in particular applicable to nonlocalized perturbations of the profiles yielding, to the author's knowledge, the first general $L^\infty$-based damping estimates for profiles of systems like \eqref{equ:general_balance_law}
(see \cite{JNRZ_14} for damping estimates on certain nonlocalized perturbations of periodic traveling-wave solutions in viscous systems of conservation laws).

On the one hand, there is an immediate application of the localized damping estimates in the study of discontinuous profiles of general hyperbolic relaxation systems \cite{Alexopoulos_Baerlin_Yang_Zumbrun_26}. The novel nonlocalized damping estimates, on the other hand, open
the door to a stability theory of nonlocalized perturbations of both constant state solutions and profiles of hyperbolic relaxation systems (see the work of Alexopoulos and deRijk for a nonlocalized stability theory of periodic wave trains
in the FitzHugh-Nagumo system \cite{Alexopoulos_deRijk_26}). Such a theory would generalize first results on nonlocalized perturbations in a $2 \times 2$ hyperbolic
relaxation system by Li and Liu \cite{Li_Liu_09} as well as a result by Bressan and collaborators \cite{Bressan_Caravenna_Shen_26} on smooth profiles of a scalar viscous conservation law. Furthermore, promising applications of the presented method include the deduction of nonlinear damping estimates
for roll waves of the inviscid Saint-Venant equations (see the work of Rodrigues and Zumbrun \cite{Rodrigues_Zumbrun_16} for linear and localized damping estimates in the mentioned setting).

\subsection{Outline of the paper}

The paper is structured as follows:

After the introductory section \ref{sec:Introduction}, Section \ref{sec:properties_of_the_profiles} discusses the central properties of the investigated profiles. The definitions of
high-frequency stable end states (Definition \ref{def:high_frequency_stability}) and of a noncharacteristic profile (Definition \ref{def:noncharacteristic}) are given. In the discontinuous case, the subshock within the profile is assumed to satisfy a stability criterion in the form of a Lopatinski condition (Definition \ref{def:Lopatinski}). The section closes by
fixing a field of left and right eigenvectors of $A(U)$ in the neighborhood of the profiles.

In Section \ref{sec:nonlocalized_damping_estimates} the nonlocalized damping estimates are proved in the smooth case, Theorem \ref{thm:damping_estimate_nonlocalized_smooth} in Subsection \ref{subsec:nonlocalized_damping_estimates_smooth}, and in the discontinuous case, Theorem \ref{thm:damping_estimate_nonlocalized_dis} in Subsection \ref{subsec:nonlocalized_damping_estimates_discontinuous}. A central Lemma \ref{lem:damping_estimate} on partial damping in
a system of ordinary differential equations is shown. The lemma is applied in the smooth as well as in the discontinuous case. While a higher-order diagonalization eliminates large off-diagonal terms in the proofs of both theorems, an important piece in the proof of the damping estimates in the discontinuous case is the bookkeeping of incoming and outgoing waves at the subshock interface
which allows to propagate the interior estimates of the incoming waves onto the outgoing waves in a Kreiss-type manner \cite{K70}.

In Section \ref{sec:localized_damping_estimates} the localized damping estimates are given. In the smooth case, discussed in Subsection \ref{subsec:localized_damping_estimates_smooth}, the proof of Theorem \ref{thm:damping_estimate_localized_smooth} makes use of similar diagonalization techniques as the ones presented in the previous Section \ref{sec:nonlocalized_damping_estimates} in combination
with a weighted energy method. Modifying the weights of the smooth and localized case to incorporate the bookkeeping on incoming and outgoing waves at the subshock interface extends the localized damping estimates of Theorem \ref{thm:damping_estimate_localized_smooth} to the discontinuous case. This is the subject of Theorem \ref{thm:damping_estimate_localized_dis} in Subsection \ref{subsec:localized_damping_estimates_dis}. The proofs of the theorems in Section
\ref{sec:localized_damping_estimates} make use of the previously obtained $L^\infty$-damping estimates of Section \ref{sec:nonlocalized_damping_estimates}. Appropriate statements on damping estimates in the localized setting that are independent of the $L^\infty$-estimates hold true at the cost of a slightly stronger assumption on the localized perturbations (see Remark \ref{rem:independent_localized_statement}).
\section{Properties of the profiles}
\label{sec:properties_of_the_profiles}

In this section the properties of the considered profiles are further specified.

We assume that \eqref{equ:general_balance_law} exhibits damping in the following sense at any one of the end states of a profile:
\begin{defi}
    \label{def:high_frequency_stability}
    Let $U^\ast \in \RR^N$ define a constant solution of \eqref{equ:general_balance_law}, i.e.~$q(U^\ast) = 0$. If there exist constants $c,C>0$ with
    \begin{align*}
        \textnormal{(S)} \quad \RE \sigma(i\xi A(U^\ast) + Q(U^\ast)) \leq -c < 0, \quad |\xi| \geq C,
    \end{align*}
    then $U^\ast$ is called (linearly) stable at high frequencies.
\end{defi}
\begin{rem}
    Assumption (S) in Definition \ref{def:high_frequency_stability} means that, after linearization at the constant solution $U^\ast$,
    the Fourier symbol of the operator in
    \eqref{equ:general_balance_law}
    has a uniform spectral gap for large frequencies.
    It is the high-frequency part of Shizuta and Kawashima's
    ``strict dissipativity'' \cite{SK85} without further symmetry assumptions.
\end{rem}
Matrix perturbation theory \cite{K95} and constant hyperbolicity of \eqref{equ:general_balance_law} imply the following useful characterization of condition (S) in
Definition \ref{def:high_frequency_stability}:
\begin{lem}
    \label{lem:negative_diagonal_blocks}
    Let $U^\ast \in \RR^N$ be a constant solution of \eqref{equ:general_balance_law}. Then $U^\ast$ is stable at high frequencies if and only if for all $j \in \lbrace 1, \ldots, J\rbrace$ and
    any choice of matrices $L_j \in \RR^{m_j \times N}$ and $R_j \in \RR^{N \times m_j}$ with $L_j R_j = I_{m_j}$ and rows, respectively columns,
    spanning the left, respectively right, kernel of $A(U^\ast) - \lambda_j(U^\ast)I_N$, the square matrix
    \[E_j = L_j Q(U^\ast) R_j \in \RR^{m_j \times m_j}\]
    has only eigenvalues with negative real parts,
    \begin{align}
        \label{inequ:E_j_negative}
        \RE \sigma (E_j) < 0.
    \end{align}
\end{lem}
\begin{rem}
    \label{rem:E_j_bound_from_frequency_bound}
    In the situation of Lemma \eqref{lem:negative_diagonal_blocks}, if $U^\ast$ is stable at high frequencies then any combination of constants $c,C > 0$ with
    \[\RE \sigma(i\xi A(U^\ast) + Q(U^\ast)) \leq -c < 0, \quad |\xi| \geq C,\]
    supplies an upper bound for the real part of the spectrum of the square matrices $E_j$ with the same constant $c$,
    \[\RE \sigma(E_j) \leq -c < 0.\]
    This can be deduced from the eigenvalue expansion up to first order of the analytic matrix family 
    $\lbrace i A(U^\ast) + \tau Q(U^\ast) \rbrace_\tau$ at $\tau = 0$ making use of the semi-simplicity of the eigenvalues
    of $A(U^\ast)$ and the correspondence $\tau = \xi^{-1}$.
\end{rem}
The damping as described in Definition \ref{def:high_frequency_stability}, or Lemma \ref{lem:negative_diagonal_blocks}, will not hold along the entire profile $\bar{U}$ but only
near its end states $U^\pm$. In order for this partial damping to be strong enough for a damping estimate along the entire profile we will assume that the considered profiles
are \textit{noncharacteristic}. This ensures that characteristics of \eqref{equ:general_balance_law} remain in the non-damping region only for a (uniform) finite time.
\begin{defi}
    \label{def:noncharacteristic}
    Let $\bar{U}$ be a profile for \eqref{equ:general_balance_law}, and denote by $\mathcal{C}$ the curve in $\RR^N$ given by the image of $\bar{U}$. Then $\bar{U}$ is called
    a noncharacteristic profile of \eqref{equ:general_balance_law} if there exists a constant $c > 0$ such that
    \begin{align*}
        \textnormal{(NC)} \quad |\lambda_j(U)| \geq c > 0 \textnormal{ for all } j \in \lbrace 1, \ldots, J \rbrace \textnormal{ and } U \in \mathcal{C}.
    \end{align*}
\end{defi}
\begin{rem}
    If $\bar{U}$ is a noncharacteristic profile of \eqref{equ:general_balance_law} then the profile equation \eqref{equ:smooth_profile}, or \eqref{equ:discontinuous_profile},
    is a regular dynamical system in a neighborhood of the curve $\mathcal{C}$ described by $\bar{U}$ where it can be equivalently written as
    \[U_x = A(U)^{-1} q(U).\]
\end{rem}

Finally, if given a profile with subshock at $x=0$ then the exchange of (characteristic) waves at the interface $x=0$ should happen in a way that allows to convert damping
estimates for incoming waves into damping estimates for outgoing waves. Hence, we will assume that the subshock is stable in the sense of Majda \cite{M83} thus satisfying the
Lopatinski condition as stated in the definition below.

\begin{defi}
    \label{def:Lopatinski}
    A pair $(U^+,U^-) \in \RR^N \times \RR^N, U^+ \neq U^-$, is said to satisfy the Lopatinski condition if
    \begin{align*}
        \textnormal{(L)} \quad \textnormal{span } \lbrace [U] \rbrace \oplus \mathbb{S}^- \oplus \mathbb{U}^+ = \RR^N
    \end{align*}
    holds with $[U] = U^+ - U^-,$ and $\mathbb{S}^\pm$ and $\mathbb{U}^\pm$ denoting the stable and unstable subspace of $A(U^\pm)$,
    respectively.
\end{defi}

We briefly discuss where to find examples of smooth and discontinuous profiles of systems like \eqref{equ:general_balance_law} satisfying Definitions \ref{def:high_frequency_stability} and \ref{def:noncharacteristic} as well as Definition \ref{def:Lopatinski} if a single subshock is present. In the small amplitude and smooth case, noncharacteristicity \ref{def:noncharacteristic} holds for relaxation profiles of Lax type (see \cite{Yong_Zumbrun_00,MZ02}). Assuming that the end states of a given profile are
high-frequency stable \ref{def:high_frequency_stability} is more restrictive since many (cf.~\cite{Yong_01,Hanouzet_Natalini_03}) but not all ``reasonable'' hyperbolic relaxation systems possess (high-frequency) stable equilibrium solutions (e.g.~\cite{Zeng_99}). The profiles
found in \cite{YZ20,Rodrigues_Yang_Zumbrun_25} for the Saint-Venant equations and the Richard-Gavrilyuk system are all noncharacteristic with high-frequency stable end states, possessing subshocks satisfying the Lopatinski condition \ref{def:Lopatinski} in the discontinuous case.

This section is closed with a last technical assumption: We assume that for all $j \in \lbrace 1, 2, \ldots, J \rbrace,$ there exist smooth matrices
$L_j=L_j(U) \in \RR^{m_j \times N}$ and $R_j=R_j(U) \in \RR^{N \times m_j}$ on $\mathcal{U}$ such that the $N \times N$ square matrices
\begin{align*}
    L(U) := \begin{pmatrix}
        L_1(U)\\
        L_2(U)\\
        \vdots\\
        L_J(U)
    \end{pmatrix}
    \quad \text{and} \quad R(U) := \begin{pmatrix}
        R_1(U) R_2(U) \cdots R_J(U)
    \end{pmatrix}
\end{align*}
are inverses of each other, $L(U) R(U) = I_N$, and diagonalize $A(U)$ with
\begin{align}
    \label{def:Lambda}
    L(U) A(U) R(U) = \Lambda(U) := \diag(\lambda_1(U)I_{m_1}, \lambda_2(U)I_{m_2}, \ldots, \lambda_J(U)I_{m_J}).
\end{align}
Also, for later use decompose
\begin{align}
    \label{def:LQR_decomposition}
    L(U) Q(U) R(U) = E(U) + F(U)
\end{align}
into a block-diagonal matrix $E(U) := \diag(E_1(U), E_2(U), \ldots, E_J(U))$ with
\begin{align}
    \label{def:E_diag}
    E_j(U) := L_j(U)Q(U)R_j(U) \in \RR^{m_j \times m_j}, j \in \lbrace 1, 2, \ldots, J \rbrace,
\end{align}
and an off-block diagonal matrix
\begin{align}
    \label{def:F_off_diag}
    F(U) := L(U)Q(U)R(U) - E(U).
\end{align}

\section{Nonlocalized damping estimates}
\label{sec:nonlocalized_damping_estimates}
The novel nonlocalized damping estimates announced in the introduction are stated and proved in the following. The smooth profile case can be found in Subsection \ref{subsec:nonlocalized_damping_estimates_smooth} while
the discontinuous case is discussed in Subsection \ref{subsec:nonlocalized_damping_estimates_discontinuous}. In both cases, damping estimates are stated in terms of norms having the form
\[\norm{f}{C^K_b(\Omega)} := \max\limits_{k \in \lbrace 0, 1, \ldots, K \rbrace} \sup\limits_{x \in \Omega} |f(x)|\]
where $\Omega \subset \RR$ is open and $f:\Omega \rightarrow \RR^M$ is a $K$-times differentiable function with integers $K = 0, 1, 2, \ldots, M = 1, 2, 3, \ldots$.
\subsection{The case of smooth profiles}
We first prove the following Lemma \ref{lem:damping_estimate}. It is used to disseminate the damping at $\pm \infty$ onto the entire profile by combining the exponential decay of the profile
to its high-frequency stable end states with noncharacteristicity.
\label{subsec:nonlocalized_damping_estimates_smooth}
\begin{lem}
    \label{lem:damping_estimate}
    Let $m \in \NN$, and let $c_E, C_d, c_d, c_{nc} > 0$ be positive constants. Then for any positive constant $0 < \theta < c_E$ there exists a constant $C>0$ such that the following holds:
    
    For $T>0$, let $E: [0,T] \times \RR \rightarrow \RR^{m \times m}$ be a continuous matrix-valued function with
    \begin{align}
        \label{est:E_decay}
        \norm{E(t,x) - E^\pm}{} \leq C_d e^{-c_d|x|}, \quad x \gtrless 0, t \in [0,T]
    \end{align}
    for limiting matrices $E^\pm \in \RR^{m \times m}$ satisfying
    \begin{align}
        \label{ieq:E_pm_negative}
        \RE \sigma(E^\pm) \leq - c_E < 0,      
    \end{align}
    and let $X: [0,T] \rightarrow \RR$ be a continuously differentiable function with
    \begin{align}
        \label{est:X_dot_bound}
        \min\limits_{t \in [0,T]}|\dot{X}(t)| \geq c_{nc} > 0.
    \end{align}
    Consider the ordinary differential equation
    \begin{align}
        \label{equ:damping_ODE}
        \dot{\Psi} = E(t,X(t)) \Psi + f(t)
    \end{align}
    where $f:[0,T] \rightarrow \RR^m$ is some continuous function. Then any smooth solution $\Psi : [0,T] \rightarrow \RR^m$ of \eqref{equ:damping_ODE} satisfies the estimate
    \begin{align}
        \label{est:damping_ODE}
        |\Psi(t)| \leq C e^{-\theta t/2} |\Psi(0)| + C \int_0^t e^{-\theta(t-s)/2}|f(s)| ds
    \end{align}
    for all $t \in [0,T]$.
\end{lem}
\begin{proof}
    Let $\theta \in (0,c_E)$. Suppose we are given a continuous matrix-valued function $E: [0,T] \times \RR \rightarrow \RR^{m \times m}$ with limiting matrices $E^\pm \in \RR^{m \times m}$
    and a continuously differentiable function $X: [0,T] \rightarrow \RR$ such that \eqref{est:E_decay}, \eqref{ieq:E_pm_negative}, and \eqref{est:X_dot_bound} hold.
    By \eqref{ieq:E_pm_negative}, we find invertible matrices $M^\pm \in \RR^{m \times m}$ such that the real, or in our case symmetric part, of
    $\tilde{E}^\pm := (M^\pm)^{-1} E^\pm M^\pm$ is negative definite with
    \begin{align}
        \label{ieq:E_pm_negative_definite}
        \RE \tilde{E}^\pm = \frac{\tilde{E}^\pm + (\tilde{E}^\pm)^\ast}{2} = \frac{\tilde{E}^\pm + (\tilde{E}^\pm)^T}{2} \leq -\theta I_m < 0.
    \end{align}
    Define
    \[C_M = \max \lbrace \norm{M^-}{},\norm{(M^-)^{-1}}{},\norm{M^+}{},\norm{(M^+)^{-1}}{}\rbrace.\]
    Let $\mathcal{F}^{s \rightarrow t}$ be the flow of the linear homogeneous equation
    \begin{align}
        \label{equ:damping_ODE_homogenuous}
        \dot{\Psi} = E(t,X(t)) \Psi \textnormal{ on } [0,T].
    \end{align}
    For a solution $\Psi$ of \eqref{equ:damping_ODE_homogenuous} set $\Psi = M^\pm \Psi^\pm$. If $X$ does not change sign on an interval $[s,t] \subset [0,T]$ then \eqref{ieq:E_pm_negative_definite} and \eqref{est:E_decay} imply
    \begin{align*}
        \frac{1}{2} \frac{d}{dr} \left| \Psi^\pm(r) \right|^2 \leq \left(- \theta + C_M^2 C_d e^{-c_d|X(r)|}\right)\left| \Psi^\pm (r)\right|^2, \quad r \in [s,t],
    \end{align*}
    reading ``$+$'' in case of $X|_{[s,t]} \geq 0$ and ``$-$'' in case of $X|_{[s,t]} \leq 0$. An application of Gronwall's Lemma shows
    \[|\Psi^\pm(t)| \leq e^{-\theta (t-s) + C_M^2 C_d \int_s^t e^{-c_d|X(r)|}dr} |\Psi^\pm(s)|\]
    which implies
    \begin{align}
        \label{est:Psi_X_constant_sign}
        |\Psi(t)| \leq C_M^2 e^{-\theta (t-s) + C_M^2 C_d \int_s^t e^{-c_d|X(r)|}dr} |\Psi(s)|.
    \end{align}
    $X$ is strictly monotonic by \eqref{est:X_dot_bound} and can therefore change sign once on a given interval $[s,t] \subset [0,T]$. Combining \eqref{est:Psi_X_constant_sign} for positive and negative sign of $X$ we find
    \begin{align}
        \label{est:Psi_X_changing_sign}
        |\Psi(t)| \leq C_M^4 e^{-\theta (t-s) + C_M^2 C_d \int_s^t e^{-c_d|X(r)|}dr} |\Psi(s)|
    \end{align}
    on any subinterval $[s,t] \subset [0,T]$. Hence, the flow of \eqref{equ:damping_ODE_homogenuous} is bounded by
    \[\norm{\mathcal{F}^{s \rightarrow t}}{} \leq C_M^4 e^{-\theta(t-s) + C_M^2 C_d \int_s^t e^{-c_d|X(r)|}dr}.\]
    Let $f: [0,T] \rightarrow \RR^m$ be a continuous function. If $\Psi$ solves the inhomogeneous equation
    \[\dot{\Psi} = E(t,X(t)) \Psi + f(t) \textnormal{ on } [0,T],\]
    then by Duhamel's formula
    \[\Psi(t) = \mathcal{F}^{0 \rightarrow t} \Psi(0) + \int_0^t \mathcal{F}^{s \rightarrow t} f(s) ds\]
    and thus
    \begin{align}
        \label{est:Psi_damping_ODE}
        |\Psi(t)| \leq C_M^4 \Big( & e^{-\theta t + C_M^2 C_d \int_0^t e^{-c_d|X(r)|}dr} |\Psi(0)|\\
        \nonumber & + \int_0^t e^{- \theta (t-s) + C_M 2C_d \int_s^t e^{-c_d|X(r)|}dr}|f(s)| ds\Big)   
    \end{align}
    for all $t \in [0,T]$.

    Set $H(t) := -\theta t + C_M^2 C_d \int_0^t e^{-c_d|X(r)|}dr$. It remains to estimate
    \begin{align}
        \label{def:exponent_damping_ODE}
        H(t) - H(s) = - \theta (t-s) + C_M^2 C_d \int_s^t e^{-c_d|X(r)|}dr
    \end{align}
    when  $0\leq s \leq t \leq T$. Since $X$ is continuously differentiable and strictly monotonic by \eqref{est:X_dot_bound} we can carry out the change of variables
    \[y= X(r), \quad dy = \dot{X} dr,\]
    in the integral on the right-hand side of \eqref{def:exponent_damping_ODE}. In combination with the fact that there exists an $R>0$ such that
    \[C_M^2 C_d e^{-c_d|y|} \leq \theta/2 \textnormal{ if } |y| \geq R,\]
    one deduces
    \begin{align}
        \label{est:exponent_damping_ODE}
        H(t) - H(s) = & -\frac{\theta}{2}(t-s) + \int_{y \in X([s,t])} \frac{-\theta/2 + C_M^2 C_d e^{-c_d|y|}}{|\dot{X}(X^{-1}(y))|}dy\\
        \nonumber \leq & -\frac{\theta}{2}(t-s) + C_M^2C_d \int_{-R}^R \frac{e^{-c_d|y|}}{c_{nc}}dy\\
        \nonumber \leq & -\frac{\theta}{2}(t-s) + \frac{4 C_M^2C_d}{c_d c_{nc}}.
    \end{align}
    Inserting inequality \eqref{est:exponent_damping_ODE} into the estimate \eqref{est:Psi_damping_ODE} yields the desired estimate \eqref{est:damping_ODE} with
    \[C:= C_M^4 e^{\frac{4 C_M^2 C_d}{c_d c_{nc}}}.\]
\end{proof}
With Lemma \ref{lem:damping_estimate} at hand we are now ready to prove the damping estimates for smooth profiles of \eqref{equ:general_balance_law}.
\begin{rem}
    The deduction is presented for smooth perturbations of the profiles.
    Formally and as customary, for a damping estimate involving $K$ derivatives, $K+1$ continuous derivatives of the perturbations are needed to make sense of the differential equation satisfied by the $K$-th derivative of the perturbations. We leave it to the reader
    to verify that for an application at hand the damping estimates in question hold if the perturbations has only $K$ continuous derivatives by invoking a local existence theory, such as the one presented in \cite{Li_Yu_85}, which allows for an approximation procedure of the 
    considered perturbations by more regular ones.
\end{rem}
\begin{thm}
    \label{thm:damping_estimate_nonlocalized_smooth}
    Let $K \in \NN$ with $K \geq 1, \tilde{R} >0$,
    and let $\bar{U}$ be a noncharacteristic smooth profile of \eqref{equ:general_balance_law} with end states $U^-, U^+$ which are stable at high frequencies.
    Then there exist constants $C, \varepsilon, \theta > 0$ such that the following holds: Let $T>0$. If $\tilde{U}$ is a smooth solution of
    \eqref{equ:general_balance_law} on $[0,T] \times \RR$ such that the shifted perturbation variable
    \[U(t,x) = \tilde{U}(t,x+\eta(t)) - \bar{U}(x)\]
    satisfies
    \begin{align*}
        \max\limits_{t \in [0,T]} \norm{U(t)}{C^1_b(\RR)} \leq \varepsilon
    \end{align*}
    and
    \begin{align*}
        \norm{U(0)}{C^{K-1}_b(\RR)} \leq \tilde{R}
    \end{align*}
    for a smooth shift $\eta: [0,T] \rightarrow \RR$ with
    \begin{align*}
        \max\limits_{t \in [0,T]} |\dot{\eta}(t)| \leq \varepsilon,
    \end{align*}
    then the damping estimate
    \begin{align}
        \label{est:damping_nonloc_smo}
        \norm{U(t)}{C^K_b(\RR)} \leq C e^{-\theta t} \norm{U(0)}{C^K_b(\RR)} + C \int_0^t e^{-\theta (t-s)}(\norm{U(s)}{C^0_b(\RR)} + |\dot{\eta}(s)|) ds
    \end{align}
    holds for all $t \in [0,T]$.
\end{thm}
\begin{proof}
    Theorem \ref{thm:damping_estimate_nonlocalized_smooth} is proved by an application of the method
of characteristics and a commutator argument inspired by stability considerations
involving Fourier symbols (e.g.~\cite{SK85}).

Let $K \in \NN$ with $K \geq 1, \tilde{R} > 0$, and let $\bar{U}$ be a noncharacteristic smooth profile of \eqref{equ:general_balance_law} with end states $U^-, U^+$ which are stable at high frequencies.
By Definitions \ref{def:smooth_profile}, \ref{def:high_frequency_stability} and \ref{def:noncharacteristic} this means that there exist constants $\theta_k, c_k, c_{nc}, c_E > 0$
such that
\begin{align}
    \label{est:profile_decay_0_nonloc_smo}
    |\partial^k_x(U(x) - U^\pm)|  & \leq c_k e^{-\theta_k |x|}, \quad x \gtrless 0, k \in \NN,\\
    \label{est:noncharacteristic_nonloc_smo}
    |\lambda_j(\bar{U}(x))| & \geq c_{nc} > 0, \quad x \in \RR, j \in \lbrace 1, \ldots, J \rbrace,\\
    \label{est:E_j_negative_nonloc_smo}
    \RE \sigma(E_j(U^\pm)) & \leq - c_E < 0, \quad j \in \lbrace 1, \ldots, J \rbrace,
\end{align}
where $E_j$ was defined in \eqref{def:E_diag}.

For some time $T>0$, suppose we are given a smooth solution $\tilde{U} \in C^\infty([0,T] \times \RR)$ of
the system of balance laws \eqref{equ:general_balance_law}, and a smooth phase shift $\eta: [0,T] \rightarrow \RR$ such that for some (small) constant $\varepsilon> 0$ 
\begin{align}
    \label{est:smallness_shift_nonloc_smo}
    \max\limits_{t \in [0,T]} |\dot{\eta}(t)| \leq \varepsilon,
\end{align}
and the shifted perturbation variable
\[U(t,x) = \tilde{U}(t,x+\eta(t)) - \bar{U}(x)\]
satisfies
\begin{align}
    \label{est:smallness_perturbation_nonloc_smo}
    \max\limits_{t \in [0,T]} \norm{U(t)}{C^1_b(\RR)} \leq \varepsilon
\end{align}
and
\begin{align}
    \label{est:boundedness_initial_perturbation}
    \norm{U(0)}{C^{K-1}_b(\RR)} \leq \tilde{R}.
\end{align}

Denote the derivative of the perturbation variable $U$ as
\begin{align*}
    W(t,x) := \partial_x U(t,x).
\end{align*}
Then $U$ and $W$ satisfy the perturbation equations
\begin{align}
    \label{U_equation}
    U_t + \left(A(\tilde{U}) - \dot{\eta}\right) U_x = & 
    Q(\tilde{U})U + \dot{\eta} \bar{U}_x\\
    \nonumber& + \left(\mathcal{O}(|\bar{U}_x|) + \mathcal{O}(|U|)\right)U
\end{align}
and 
\begin{align}
    \label{W_equation}
    W_t + \left(A(\tilde{U}) - \dot{\eta}\right) W_x = & 
    Q(\tilde{U})W + \dot{\eta} \bar{U}_{xx}\\
    \nonumber& + \left(\mathcal{O}(|W|)+\mathcal{O}(|\bar{U}_x|)\right)W\\
    \nonumber& + \left(\mathcal{O}(|\bar{U}_{xx}|) +
    \mathcal{O}(|\bar{U}_x|)\right) U
\end{align}
where we abused notation by setting
\[A(\tilde{U})(t,x) = A(\tilde{U}(t,x+\eta(t))).\]
Furthermore, a term of the form $\mathcal{O}(f(t,x))$ with $f=f(t,x) \geq 0$ denotes a generic function $g = g(t,x)$ such that for all $(t,x)$ under consideration the estimate
\[|g(t,x)| \leq C f(t,x)\]
holds for some constant $C$ depending on $\varepsilon$, $\tilde{R}$, and the size of (the derivatives of) $\bar{U}$ and of (the derivatives of) the coefficients in \eqref{equ:general_balance_law} alone.

\textbf{The method of characteristics:}
In order to apply the method of characteristics, the differential operator on
the left-hand sides of \eqref{U_equation} and \eqref{W_equation} needs
to be diagonalized. Thus, set
\begin{align*}
    \Phi(t,x) := L(\tilde{U}(t,x+\eta(t))) U(t,x) \text{ and }
    \Psi(t,x) := L(\tilde{U}(t,x+\eta(t))) W(t,x).
\end{align*}
Performing this change of variables, the equation satisfied by $\Phi$ is
\begin{align}
    \label{Phi_equation}
    \Phi_t + D(t,x) \Phi_x = \mathcal{O}(|\Phi|) + \mathcal{O}(|\dot{\eta}|)
\end{align}
with diagonal matrix
\begin{align*}
    D(t,x) := \Lambda(\tilde{U}(t,x+\eta(t))) - \dot{\eta}(t)
\end{align*}
where $\Lambda$ was defined in \eqref{def:Lambda}.

For $j \in \lbrace 1, \ldots, J \rbrace$ denote by $\Phi_j= \Phi_j(t,x) \in \RR^{m_j}$ the $j$-th component of $\Phi$ according to
\[\Phi(t,x) = \begin{pmatrix}
    \Phi_1(t,x) \\
    \Phi_2(t,x) \\
    \vdots \\
    \Phi_J(t,x)
\end{pmatrix}.
\]
Then $\Phi_j$ satisfies
\begin{align}
    \label{Phi_equation_j}
    \partial_t \Phi_j + D_{j}(t,x)\partial_x \Phi_j = \mathcal{O}(|\Phi|) + \mathcal{O}(|\dot{\eta}|).
\end{align}
where
\[D_j(t,x) := (\lambda_j(\tilde{U}(t,x+\eta(t))) -\dot{\eta}(t)) I_{m_j}.\]
A ``damping estimate'' for $\Phi$ is easily derived from \eqref{Phi_equation}, by adding and subtracting (any) constant $\theta > 0$ of $\Phi$ in \eqref{Phi_equation}
and integrating the components of $\Phi$ along corresponding characteristics:

For $(t,x) \in [0,T] \times \RR$, and $j \in \lbrace 1, \ldots, J
\rbrace$, the $j$-th characteristic curve through $(t,x)$ is given by the trace of the map $(s \mapsto X_j(s))$ with $X_j=X_j(s)=X_j(s;(t,x))$ solving the ordinary differential equation
\begin{gather}
        \label{definition_j_X}
        \begin{aligned}
            \dot{X}_j(s) & = \lambda_j(\tilde{U}(s,x+\eta(s))) - \dot{\eta}(s),\\
            X_j(t) &= x.
        \end{aligned}
    \end{gather}
on $[0,T]$. 

Choose any $\theta > 0$. Along the $j$-th characteristic the $j$-th equation \eqref{Phi_equation_j} reads
\begin{align*}
    \dot{\Phi}_j(s) = & \mathcal{O}(|\Phi(s,X_j(s))|) + \mathcal{O}(|\dot{\eta}(s)|) \\
    = & - \theta \Phi_j(s) + (\theta \Phi_j(s) +  \mathcal{O}(|\Phi(s,X_j(s))|) + \mathcal{O}(|\dot{\eta}(s)|))
\end{align*}
where we have used the notation
\[\Phi_j(s) = \Phi_j(s,X_j(s)) \textnormal{ and } \dot{\Phi}_j(s) = \frac{d \Phi_j}{ds}(s).\]
Using Duhamel's formula, one obtains the estimate
\begin{align}
    \label{est:Phi_j_theta}
    |\Phi_j(t,x)| \leq & e^{-\theta t} |\Phi_j(0,x_0)| + C(1 + \theta) \int_0^t e^{-\theta (t-s)} (|\Phi(s, X_j(s))| + |\dot{\eta}(s)|)ds\\
    \nonumber \leq & e^{-\theta t} \norm{\Phi(0)}{C^0_b}  + C(1 + \theta) \int_0^t e^{-\theta (t-s)} (\norm{\Phi(s)}{C^0_b}  + |\dot{\eta}(s)|)ds
\end{align}
with $X_j$ defining the $j$-th characteristic through $(t,x)$ and $x_0 =X_j(0)$.

Taking the supremum over $x$ in \eqref{est:Phi_j_theta}, and collecting all components of $\Phi$ gives
\begin{align}
    \label{Phi_estimate_final}
    \norm{\Phi(t)}{C^0_b} \leq C e^{-\theta t} \norm{\Phi(0)}{C^0_b} + C (1 + \theta)
    \int_0^t e^{-\theta(t-s)} \left(\norm{\Phi(s)}{C^0_b} +
    |\dot{\eta}(s)|\right)ds
\end{align}
for any $\theta > 0$ and some uniform $C > 0$.

\textbf{The damping estimate:}

Our next goal is to show a damping estimate on $\Psi$ that
``binds'' the $C^0_b$-norm of $\Psi$ to the initial values $\Phi(0), \Psi(0)$ and
a Duhamel term involving only $U$ and $\dot{\eta}$ but not $\Psi$ itself.
The estimate will be of the form
\begin{align}
    \label{Psi_damping_estimate}
    \norm{\Psi(t)}{C^0_b} \leq C e^{-\frac{\theta}{2} t} \norm{(\Phi(0),\Psi(0))}{C^0_b} + C
    \int_0^t e^{-\frac{\theta}{2}(t-s)} \left(\norm{\Phi(s)}{C^0_b} +
    |\dot{\eta}(s)|\right)ds
\end{align}
with constants $\theta, C > 0$ depending on $\varepsilon$, and the properties of the profile
$\bar{U}$ and the properties of the coefficients appearing in \eqref{equ:general_balance_law} alone.

To achieve an estimate like \eqref{Psi_damping_estimate} we cannot simply work with multiples of $\Psi$
as we did in the previous deduction of the estimate \eqref{Phi_estimate_final} of
$\Phi$ where we accepted $\Phi$ itself as part of the linear source term. We
will eliminate all (large) off-diagonal terms in the (linear) $\Psi$-source term by a variable transform involving $\Psi$ and $\Phi$
at the cost of increasing the $\Phi$-source term. But since 
we treat a term of order $\mathcal{O}(\Phi)$ as a forcing, in 
accordance with the damping estimate \eqref{Psi_damping_estimate}, 
the resulting larger $\Phi$-source term does not pose a problem.

Let us first state an equation for $\Psi$. Replacing $W$ by
$R(\tilde{U}) \Psi$ in equation \eqref{W_equation} and multiplying both sides
by $L(\tilde{U})$ from the left yields
\begin{align}
    \label{Psi_equation}
    \Psi_t + D(t,x) \Psi_x = & \lbrace E(\bar{U}) + F(\bar{U}) \rbrace \Psi + 
    \mathcal{O}(|\bar{U}_x|)\Psi + \mathcal{O}(\varepsilon)\Psi\\
    \nonumber & + \mathcal{O}(|\Phi|) + \mathcal{O}(|\dot{\eta}|).
\end{align}
We point out that no derivatives of $\Psi$ appear on the right-hand side of
\eqref{Psi_equation}, thus the variable change does not introduce higher
derivatives of $U$ other than $U_x = W$ into the equation for $\Psi$. The term
\[\mathcal{O}(|\Phi|) + \mathcal{O}(|\dot{\eta}|)\]
will be treated as a forcing, while the cross-term
\[\mathcal{O}(\varepsilon) \Psi\]
is viewed as a higher order term that can be bounded uniformly by smallness of $\varepsilon$.

To make sure that only the forcing term
and the cross-term appear in the slaving of $\Psi$, 
we need to diagonalize the linear source term
\[\lbrace E(\bar{U}) + F(\bar{U}) \rbrace \Psi + \mathcal{O}(|\bar{U}_x|)\Psi.\]
To this end, decompose the $\mathcal{O}(|\bar{U}_x|)$ term into diagonal and
off-diagonal parts, too, and collect to find
\[\lbrace E(\bar{U}) + F(\bar{U}) \rbrace \Psi + \mathcal{O}(|\bar{U}_x|)\Psi = \lbrace \tilde{E}(t,x) + \tilde{F}(t,x) \rbrace\Psi\]
where
\begin{align*}
    \tilde{E} = E(\bar{U}) + \mathcal{O}(|\bar{U}_x|) \quad \text{and} \quad
    \tilde{F} = F(\tilde{U}) + \mathcal{O}(|\bar{U}_x|)
\end{align*}
do not contain terms that involve $\Psi$ and $\dot{\eta}$. The $t$-dependence of $\tilde{E}$ and $\tilde{F}$ comes from $t$-dependent factors in the $\mathcal{O}(|\bar{U}_x|)$ term. By choosing a suitable combination of $\Psi$ and $\Phi$ we will eliminate the
off-diagonal terms $\tilde{F} \Psi$ in the equation \eqref{Psi_equation} while
preserving its structure.

For a matrix $\Theta = \Theta(t,x) \in
\RR^{N \times N}$ that is to be determined, set
\begin{align}
    \label{def:tilde_Psi}
    \tilde{\Psi} = \Psi + \Theta \Phi.
\end{align}
As we shall see later, we may assume boundedness of $\Theta$:
\[\Theta(t,x) = \mathcal{O}(1).\]
For the $x$-derivative of $\tilde{\Psi}$ we have
\begin{align}
    \label{Psi_tilde_x_derivative}
    \tilde{\Psi}_x & = \Psi_x + \Theta \Phi_x + \Theta_x \Phi\\
    \nonumber& = \Psi_x + \Theta \Psi + \Theta \mathcal{O}(|\Phi|) + \Theta_x \Phi
\end{align}
Since we want to treat any term that involves a multiple of (a component of) $\Phi$
as a large, but bounded forcing, $\Theta$ can only be constructed from terms 
involving $\Phi$ and $\bar{U}$ but not $\Psi$, otherwise we lose control of the derivatives $\Theta_x$ 
and $\Theta_t$ (we handle time derivatives of $\Phi$ using the equation \eqref{Phi_equation}
satisfied by $\Phi_t$.)

For the time derivative of $\tilde{\Psi}$ we find from the equations satisfied by $\Psi$ and $\Phi$, equations \eqref{Phi_equation} and \eqref{Psi_equation}, respectively, that
\begin{align}
    \label{Psi_tilde_equation_1}
    \tilde{\Psi}_t = & \Psi_t + \Theta \Phi_t + \Theta_t \Phi\\
    \nonumber= & -D(t,x) \Psi_x + \lbrace \tilde{E}(t,x) + \tilde{F}(t,x) \rbrace \Psi
    + \mathcal{O}(\varepsilon)\Psi\\
    \nonumber& - \Theta(t,x) D(t,x) \Phi_x + \mathcal{O}(|\dot{\eta}|) + \mathcal{O}(|\Phi|) + \Theta_t \Phi.
\end{align}
Solving the equation \eqref{Psi_tilde_x_derivative} for $\Psi_x$ and inserting the result in \eqref{Psi_tilde_equation_1} as well as using
\[\Phi_x = \Psi + \mathcal{O}(|\Phi|)\]
yields
\begin{align}
    \label{Psi_tilde_equation_2}
    \tilde{\Psi}_t = & -D(t,x) \tilde{\Psi}_x + \tilde{E}(t,x) \tilde{\Psi}
    + \mathcal{O}(\varepsilon) \tilde{\Psi}\\
    \nonumber& + \big[D(t,x) \Theta(t,x) + \tilde{F}(t,x) - \Theta(t,x) D(t,x) \big]\Psi \\
    \nonumber& + \mathcal{O}(|\dot{\eta}|) + \mathcal{O}(|\Phi|) + \Theta_t \Phi + \mathcal{O}(1) \Theta_x \Phi.
\end{align}
Hence, if $\Theta$ solves
\begin{align}
    \label{Theta_defining_equation}
    \Theta D - D\Theta = [\Theta, \Lambda] =\tilde{F}
\end{align}
then equation \eqref{Psi_tilde_equation_2} will be free of linear off-diagonal $\tilde{\Psi}$
terms (see Remark \ref{conjugation_remark} for a discussion of the origins
of a similar commutator argument in Fourier analysis). Inspecting each component, we find that the solution of equation \eqref{Theta_defining_equation} is
\begin{align}
    \label{Theta_definition}
    \Theta(t,x)_{jk} := \begin{cases}
    \frac{\tilde{F}(t,x)_{jk}}{\lambda_k(\tilde{U}(t,x+\eta(t)))-\lambda_j(\tilde{U}(t,x+\eta(t)))} & j \neq k,\\
    0_{m_j \times m_j} & j=k,
    \end{cases}
\end{align}
for $j,k \in \lbrace 1, \ldots, J \rbrace, (t,x) \in [0,T] \times \RR,$ with rectangular matrices
\[\Theta(t,x)_{jk}, \tilde{F}(t,x)_{jk} \in \RR^{m_j \times m_k}.\]
Since $\tilde{F}$ is free of contributions from $\Psi$ and $\dot{\eta}$, so is
$\Theta$, and we find
\[\Theta(t,x) = \mathcal{O}(1), \quad \Theta_t \Phi = \mathcal{O}(|\Phi|), \quad \Theta_x \Phi = \mathcal{O}(|\Phi|).\]
With this choice of $\Theta$ the equation for $\tilde{\Psi}$ reads
\begin{align}
    \label{Psi_tilde_equation_final}
    \tilde{\Psi}_t +  D(t,x) \tilde{\Psi}_x = & \tilde{E}(t,x) \tilde{\Psi}
    + \mathcal{O}(\varepsilon)\tilde{\Psi} + \mathcal{O}(|\dot{\eta}|) + \mathcal{O}(|\Phi|).
\end{align}

For $(t,x) \in [0,T] \times \RR$ and $j \in \lbrace 1, \ldots, J \rbrace$, the $j$-th component $\tilde{\Psi}_j= \tilde{\Psi}_j(t,x) \in \RR^{m_j}$ of $\tilde{\Psi}$
satisfies
\begin{align}
    \label{equ:Psi_tilde_j_equ}
    \dot{\tilde{\Psi}}_j(s) = & \tilde{E}(s,X_j(s))_j \tilde{\Psi}_j(s) + \mathcal{O}(\varepsilon|\tilde{\Psi}(s,X_j(s))|) \\ 
    \nonumber & + \mathcal{O}(|\dot{\eta}(s)|) + \mathcal{O}(|\Phi(s,X_j(s)|)
\end{align}
along the $j$-th characteristic through $(t,x)$. For $\varepsilon$ small enough, the properties of the profile $\bar{U}$, as summarized in \eqref{est:profile_decay_0_nonloc_smo}
to \eqref{est:E_j_negative_nonloc_smo}, imply the existence of constants $C_d, c_d > 0$ such that for $E^\pm_j:= E_j(U^\pm)$ it holds
\[\RE \sigma(E^\pm_j) \leq - c_E < 0\]
and
\[|\tilde{E}(s,x)_j - E^\pm_j| \leq C_d e^{-c_d|x|}, \quad x \gtrless 0, s \in [0,T]\]
as well as
\[\min\limits_{s \in [0,T]} |\dot{X}_j(s)| \geq \frac{c_{nc}}{2} > 0, \]
where
\[\dot{X}_j(s) = \lambda_j(\tilde{U}(s,X_j(s)+\eta(s))) - \dot{\eta}(s) = \lambda(\bar{U}(X_j(s))) + \mathcal{O}(\varepsilon)\]
was used.

Choose a constant $\theta > 0$ with $2 \theta < c_E$. By Lemma \ref{lem:damping_estimate} there exists a constant $C > 0$ such that
\begin{align*}
    |\tilde{\Psi}_j(t,x)| \leq & C e^{-\theta t} \norm{\tilde{\Psi}(0)}{C^0_b} + C \int_0^t e^{-\theta(t-s)} (\norm{\Phi(s)}{C^0_b} + |\dot{\eta}(s)|) ds\\
    \nonumber & + C \varepsilon \int_0^t e^{-\theta(t-s)} \norm{\tilde{\Psi}(s)}{C^0_b})ds
\end{align*}
Again, taking the supremum in $x$ on the left-hand side of the previous estimate, and collecting all components of $\tilde{\Psi}$ yields the estimate
\begin{align}
    \label{Psi_estimate_pre_final}
    \norm{\tilde{\Psi}(t)}{C^0_b} \leq &C e^{-\theta t} \norm{\tilde{\Psi}(0)}{C^0_b} + C
    \int_0^t e^{-\theta(t-s)} \left(\norm{\Phi(s)}{C^0_b} +
    |\dot{\eta}(s)|\right)ds\\
    \nonumber& + C \varepsilon \int_0^t e^{-\theta(t-s)} \norm{\tilde{\Psi}(s)}{C^0_b}ds.
\end{align}
The cross-term $ \varepsilon \tilde{\Psi}$ in \eqref{Psi_estimate_pre_final} can be eliminated by Gronwall's Lemma if $\varepsilon \leq \theta (2C)^{-1}$ without changing the constant $C>0$.
We arrive at the desired estimate
\begin{align}
    \label{Psi_estimate_final}
    \norm{\tilde{\Psi}(t)}{C^0_b} \leq C e^{-\frac{\theta}{2} t} \norm{\tilde{\Psi}(0)}{C^0_b} + C
    \int_0^t e^{-\frac{\theta}{2}(t-s)} \left(\norm{\Phi(s)}{C^0_b} +
    |\dot{\eta}(s)|\right)ds.
\end{align}
The definition of $\tilde{\Psi}= \Psi + \Theta \Phi$ with bounded
$\Theta$ and the $\Phi$-estimate \eqref{Phi_estimate_final} with $\theta/2 >0$ as in \eqref{Psi_estimate_final} imply the damping
estimate \eqref{Psi_damping_estimate} on $\Psi$.

\begin{rem}
    \label{conjugation_remark} The application of the variable transform
    \eqref{definition_tilde_Upsilon} to eliminate higher order off-diagonal
    terms by solving a commutator equation like \eqref{Theta_defining_equation}
    is inspired by stability considerations regarding the linearization of
    \eqref{equ:general_balance_law} at a (constant)
    equilibrium solution. In \cite{SK85} the analog of \eqref{Theta_defining_equation}
    in Fourier space is used to prove various characterizations of (strict) dissipativity
    for systems like \eqref{equ:general_balance_law} resulting
    in theorems regarding the global existence of solutions to
    \eqref{equ:general_balance_law} for data near equilibria.
    These insights and results have been generalized in the following years, e.g.~\cite{Zeng_99,Yong_01,Hanouzet_Natalini_03,KY04,BHN07,BZ11,S25}.
\end{rem}
\begin{rem}
    Non-characteristicity is used in an essential way to derive the damping 
    estimate \eqref{Phi_estimate_final}. The idea is that each characteristic
    stays in a region with no damping, quantified as $[0,T] \times [-R,R]$ in Lemma \ref{lem:damping_estimate},
    only for a finite time which results in a contribution to the constants $C$ in
    \eqref{Phi_estimate_final} but leaves enough time for ``overall'' damping due to the exponential
    decay of the profile to the end states $U^\pm$ where damping is present.
\end{rem}
\textbf{Damping estimates for higher derivatives:}

The derivation of a damping estimate for higher derivatives $\partial_x^K U, K \geq 2,$ does not require a smallness
assumption on derivatives of $\partial_x^k U, k > 1,$ if we demand $\varepsilon < \theta / K$ (see the comment at the end of the proof). We demonstrate the extension of the damping estimate to the
second derivative
\[Y(t,x) := U_{xx}.\]
Then $Y$ solves
\begin{align}
    \label{Y-equation}
    Y_t + (A(\tilde{U})-\dot{\eta}) Y_x = & Q(\tilde{U}) Y + \left(\mathcal{O}(|W|) + \mathcal{O}(|\bar{U}_x|)\right)Y + \dot{\eta} \bar{U}_{xxx} \\
    \nonumber & + \mathcal{O}(|W|) \left(\mathcal{O}(|W|) + \mathcal{O}(|W|^2)\right.\\
    \nonumber & \quad\quad\quad\quad\quad \left. + \mathcal{O}(|\bar{U}_x|)
    + \mathcal{O}(|\bar{U}_{xx}|) + \mathcal{O}(|\bar{U}_x|^2)\right)\\
    \nonumber & + \mathcal{O}(|U|)\left(\mathcal{O}(|\bar{U}_{xx}|) + \mathcal{O}(|\bar{U}_{x}|^2) + \mathcal{O}(|\bar{U}_{x}||\bar{U}_{xx}|)\right.\\
    \nonumber & \quad\quad\quad\quad\quad  \left. + \mathcal{O}(|\bar{U}_{x}|^3)\right)
\end{align}

\begin{rem}
    \label{remark_higher_order_1}
    The form of equation \eqref{Y-equation} is typical for derivatives of $U$ of order
    two or higher: No higher powers of components of $Y$ appear,
    that is, there are no $\mathcal{O}(|Y|^2)$-terms. In fact, by the product rule, the relevant part
    for the damping estimate will always have the form
    \[ Y_t + (A(\tilde{U})-\dot{\eta}) Y_x = Q(\tilde{U}) Y + \left(\mathcal{O}(|W|) + \mathcal{O}(|\bar{U}_x|)\right)Y + \dot{\eta} \bar{U}_{xxx} + \ldots\]
    where $Y$ does not appear in any of the trailing terms indicated by ``$\ldots$''.
\end{rem}

As before, the change of variables
\begin{align*}
    \Upsilon(t,x) = L(\tilde{U}(t,x+\eta(t)))Y(t,x)
\end{align*}
is applied, and the equation for $\Upsilon$ is derived:
\begin{align}
    \label{Upsilon_equation}
    \Upsilon_t + D(t,x) \Upsilon_x = & \lbrace E(\bar{U}) + F(\bar{U}) \rbrace \Upsilon + 
    \mathcal{O}(|\bar{U}_x|)\Upsilon + \mathcal{O}(\varepsilon) \Upsilon\\
    \nonumber & + \mathcal{O}(|\dot{\eta}|)+ \mathcal{O}(|\Phi|)+ \mathcal{O}(|\Psi|).
\end{align}
Notice that compared to \eqref{Psi_equation} now $\Psi$ appears as a forcing, too.
Proceeding as in the previous section, we consider
\begin{align}
    \label{definition_tilde_Upsilon}
    \tilde{\Upsilon} = \Upsilon + \Theta \Psi
\end{align}
where $\Theta$ is exactly as in \eqref{Theta_definition} (this is because 
$\Upsilon$ has replaced $\Psi$ as the highest derivative of $U$, and the
deductions of \eqref{Upsilon_equation} and \eqref{Psi_equation} follow 
the same lines). Definition \eqref{definition_tilde_Upsilon} again eliminates
problematic off-diagonal terms in \eqref{Upsilon_equation}, and we find the
damping estimate
\begin{align}
    \label{Upsilon_damping_estimate}
    \norm{\Upsilon(t)}{C^0_b} \leq & C e^{-\frac{\theta}{2} t} \norm{(\Phi(0),\Psi(0), \Upsilon(0))}{C^0_b}\\
    \nonumber & + C \int_0^t e^{-\frac{\theta}{2}(t-s)} \left(\norm{\Phi(s)}{C^0_b} + \norm{\Psi(s)}{C^0_b}
    + |\dot{\eta}(s)|\right)ds
\end{align}
for $\Upsilon$ with a possibly enlarged constant $C > 0$.

It remains to remove the additional integral of $\Psi$ in \eqref{Upsilon_damping_estimate}.
This is done by a standard Gronwall-type estimate using the damping estimate on
$\Psi$. That is, by setting
\[\mathbb{V}_0 := \norm{(\Phi(0),\Psi(0),\Upsilon(0))}{C^0_b}, \quad \mathbb{G}(t) := \int_0^t e^{\frac{\theta}{2} s}\left(\norm{\Phi(s)}{C^0_b} + |\dot{\eta}(s)|\right) ds,\]
and combining \eqref{Upsilon_damping_estimate}
and \eqref{Psi_damping_estimate} we find for any positive integer $N>0$
\begin{align}
    \label{remove_Psi_integral_1}
    \norm{\Upsilon(t)}{C^0_b} + N \norm{\Psi(t)}{C^0_b}\leq &C(1+N) e^{-\frac{\theta}{2} t} (\mathbb{V}_0 + \mathbb{G}(t))\\
    \nonumber& + \frac{C}{N} e^{-\frac{\theta}{2} t} \int_0^t e^{\frac{\theta}{2} s}\left(\norm{\Upsilon(s)}{C^0_b} + N \norm{\Psi(s)}{C^0_b}\right) ds.
\end{align}
With 
\[\mathbb{V}_N(t) := e^{\frac{\theta}{2} t} \left(\norm{\Upsilon(t)}{C^0_b} + N \norm{\Psi(t)}{C^0_b}\right) \quad (N>0)\]
we see from \eqref{remove_Psi_integral_1}
that
\begin{align}
    \label{V_inequality_1}
    \mathbb{V}_N(t) \leq C(1+N)\left(\mathbb{V}_0 + \mathbb{G}(t)\right) + \frac{C}{N}\int_0^t \mathbb{V}_N(t).
\end{align}
Applying Gronwall's Lemma to \eqref{V_inequality_1} gives
\begin{align}
    \label{V_inequality_2}
    \mathbb{V}_N(t) \leq C(1+N)\Big[& \Big(1+e^{\frac{C}{N}t}\int\limits_0^t \frac{C}{N}e^{-\frac{C}{N}s}ds\Big)\mathbb{V}_0\\
    \nonumber& + \mathbb{G}(t) + e^{\frac{C}{N}t}\int\limits_0^t \frac{C}{N}e^{-\frac{C}{N}s}\mathbb{G}(s)ds \Big].
\end{align}
Integrating the first integral in \eqref{V_inequality_2} and using partial integration to
treat the second integral in \eqref{V_inequality_2} yields
\begin{align}
    \label{V_inequality_3}
    \mathbb{V}_N(t) \leq C(1+N)\Big[e^{\frac{C}{N}t}\mathbb{V}_0 + \int\limits_0^t e^{\frac{C}{N}(t-s)}e^{\frac{\theta}{2} s}\left(\norm{\Phi(s)}{C^0_b} + |\dot{\eta}(s)|\right)ds\Big].
\end{align}
Multiply \eqref{V_inequality_3} by $e^{-\frac{\theta}{2} t}$ and unravel the definitions
of $\mathbb{V}_N, \mathbb{V}_0$ and $G$ to find
\begin{align*}
    \norm{\Upsilon(t)}{C^0_b}\leq & \norm{\Upsilon(t)}{C^0_b} + N \norm{\Psi(t)}{C^0_b}\\
    \nonumber \leq & C(1+N) e^{(C/N-\frac{\theta}{2})t}\norm{(\Phi(0),\Psi(0),\Upsilon(0))}{C^0_b}\\
    \nonumber & + C(1+N) \int\limits_0^t e^{(C/N-\frac{\theta}{2})(t-s)}\left(\norm{\Phi(s)}{C^0_b} + |\dot{\eta}(s)|\right)ds.
\end{align*}
Choosing $N$ large enough so that
\[\frac{4C}{\theta} < N \leq \frac{4C}{\theta} + 1\]
gives the desired damping estimate
\begin{align}
    \label{Upsilon_damping_estimate_final}
    \norm{\Upsilon(t)}{C^0_b} \leq & C e^{-\frac{\theta}{4} t} \norm{(\Phi(0),\Psi(0), \Upsilon(0))}{C^0_b}\\
    \nonumber & + C \int_0^t e^{-\frac{\theta}{4}(t-s)} \left(\norm{\Phi(s)}{C^0_b} + |\dot{\eta}(s)|\right)ds
\end{align}
when enlarging the constant $C$ according to $C \leadsto C(2+4C/\theta)$.
Since the change of variables
\[(t,x) \mapsto L(\tilde{U}(t,x+\eta(t))), \quad (t,x) \mapsto R(\tilde{U}(t,x+\eta(t))) \]
is uniformly bounded by the decay of the profile $\bar{U}$ at infinity \eqref{est:profile_decay_0_nonloc_smo} and smallness of the perturbation $U$ \eqref{est:smallness_perturbation_nonloc_smo}
we can summarize our estimates \eqref{Phi_estimate_final}, \eqref{Psi_damping_estimate}, and \eqref{Upsilon_damping_estimate_final} in the inequality
\begin{align*}
    \norm{U(t)}{C^2_b} \leq Ce^{-\frac{\theta}{4} t} \norm{U(0)}{C^2_b} + C \int\limits_0^t e^{-\frac{\theta}{4}(t-s)}\left(\norm{U(s)}{C^0_b} + |\dot{\eta}(s)|\right)ds.
\end{align*}
where $C>0$ is taken appropriately large. This proves Theorem \ref{thm:damping_estimate_nonlocalized_smooth} for $K = 1, 2$.

For $K > 2$ damping estimates read
\begin{align*}
    \norm{U(t)}{C^K_b} \leq Ce^{-\frac{\theta}{2K} t} \norm{U(0)}{C^K_b} + C \int\limits_0^t e^{-\frac{\theta}{2K}(t-s)}\left(\norm{U(s)}{C^0_b} + |\dot{\eta}(s)|\right)ds.
\end{align*}
While following the above-described route, taking care of additional forcings as done for the $\Upsilon$-estimate, or $U_{xx}$-estimate, one needs to estimate products consisting of derivatives of $U$ 
up to order $K-1$. Notice that order $K$-derivatives are only paired with derivatives of $U$ that have order smaller or equal to $1$ which are of order $\varepsilon$ by assumption
\eqref{est:smallness_perturbation_nonloc_smo} (see Remark \ref{remark_higher_order_1}). lower-order products are taken care of by estimating all but one of the factors in those products by the fact that a previously
shown damping estimate for derivatives of $U$ of order $K-1$ implies the bound
\begin{align*}
    \norm{U(t)}{C^{K-1}_b} \leq C \tilde{R} + 4C \frac{\varepsilon(K-1)}{\theta}
\end{align*}
where the contribution from the initial data was bounded by assumption \eqref{est:boundedness_initial_perturbation}. With the choice $\varepsilon < \theta/K$ the lower-order products can be treated as 
forcing again at the cost of increasing the constant $C> 0$ by a factor of $\tilde{R} + 4$.

\end{proof}
\subsection{The case of discontinuous profiles}
\label{subsec:nonlocalized_damping_estimates_discontinuous}
We turn to discontinuous profiles. When deducing the damping estimates for profiles with a subshock, the previous insights in the smooth profile case are used to
prove partial ``interior'' damping estimates which are then combined using a Rankine-Hugoniot condition and the Lopatinski condition (Definition \ref{def:Lopatinski}) to obtain the full damping estimates.

For a concise presentation, let us fix some notation regarding solutions of \eqref{equ:general_balance_law} with a single discontinuity:
\begin{defi}
    \label{def:solutions_with_subshock}
    Let $T> 0$. A function $\tilde{U}=\tilde{U}(t,x)$ is called a solution of \eqref{equ:general_balance_law} on $[0,T]$ with subshock at $\eta$
    if $\eta:[0,T] \rightarrow \RR$ is a smooth function such that $\tilde{U}$ is a smooth solution of \eqref{equ:general_balance_law} on
    $([0,T] \times \RR) \backslash \lbrace x=\eta(t) \rbrace$ with a single jump across $\lbrace x=\eta(t) \rbrace$ where the
    jump condition
    \begin{align*}
        \dot{\eta}(t)[\tilde{U}(t)]_{\eta(t)} = [F(U(t))]_{\eta(t)}
    \end{align*}
    with well-defined one-sided limits $\lim_{x \rightarrow \eta(t)^\pm} U(t,x)= U(t,\eta(t)^\pm)$ holds for all $t \in [0,T]$.
\end{defi}

For the existence of solutions of \eqref{equ:general_balance_law} like the ones in Definition \ref{def:solutions_with_subshock} we refer the reader to the work of Ta-tsien Li and Wen-ci Yu \cite{Li_Yu_85} (see \cite{M83,M01} for an $L^2$-based existence theory).

The central result of this subsection is:

\begin{thm}
    \label{thm:damping_estimate_nonlocalized_dis}
    Let $K \in \NN$ with $K \geq 1, \tilde{R}> 0$, and let $\bar{U}$ be a noncharacteristic discontinuous profile of \eqref{equ:general_balance_law} with end states $U^-, U^+$
    which are stable at high frequencies. Assume that $(\bar{U}(0^+), \bar{U}(0^-))$ satisfies the Lopatinski condition.
    Then there exist constants $C, \varepsilon, \theta > 0$ such that the following holds:
    
    Let $T>0$. If $\tilde{U}$ is a solution of
    \eqref{equ:general_balance_law} on $[0,T]$ with subshock at $\eta$ such that
    \begin{align*}
        \max\limits_{t \in [0,T]} |\dot{\eta}(t)| \leq \varepsilon,
    \end{align*}
    and the shifted perturbation variable
    \[U(t,x) = \tilde{U}(t,x+\eta(t)) - \bar{U}(x)\]
    satisfies
    \begin{align*}
        \max\limits_{t \in [0,T]} \norm{U(t)}{C^1_b(\dot{\RR})} \leq \varepsilon,
    \end{align*}
    and
    \begin{align*}
        \norm{U(0)}{C^{K-1}_b(\dot{\RR})} \leq \tilde{R}
    \end{align*}
    then the damping estimate
    \begin{align}
        \label{est:damping_nonloc_dis}
        \norm{U(t)}{C^K_b(\dot{\RR})} \leq C e^{-\theta t} \norm{U(0)}{C^K_b(\dot{\RR})} + C \int_0^t e^{-\theta (t-s)}(\norm{U(s)}{C^0_b(\dot{\RR})} + |\dot{\eta}(s)|) ds
    \end{align}
    holds for all $t \in [0,T]$.
\end{thm}
\begin{proof}
    Theorem \ref{thm:damping_estimate_nonlocalized_dis} is proved by making use of the previously obtained estimates in the smooth profile case, 
and a Kreiss-type estimate that provides control of the perturbations on the boundary $[0,T] \times \lbrace 0 \rbrace$.

Let $K \in \NN$ with $K \geq 1, \tilde{R} > 0$, and let $\bar{U}$ be a noncharacteristic discontinuous profile of \eqref{equ:general_balance_law} with end states $U^-, U^+$
which are stable at high frequencies. Assume that the pair $(\bar{U}(0^+), \bar{U}(0^-))$ satisfies the Lopatinski condition as in Definition \ref{def:Lopatinski},
\begin{align}
    \label{equ:Lopatinski_condition}
    [U] \oplus \mathbb{S}(A(\bar{U}(0^-))) \oplus \mathbb{U}(A(\bar{U}(0^+)))= \RR^N,
\end{align}
which implies the existence of $p \in \lbrace 1, 2, \ldots, N \rbrace$ and $j^\pm \in \lbrace 1, 2, \ldots, J \rbrace$ such that
\[\sum_{j=1}^{j^-}m_j = \dim(\mathbb{S}(A(\bar{U}(0^-)))) = p-1\]
and
\[\sum_{j=j^+}^J m_j = \mathbb{U}(A(\bar{U}(0^+))) = N-p.\]
Note that $j^\prime := j^- + 1 = j^+ -1$. The relation $m_{j^\prime} = 1$ follows from $\sum_{j=1}^J m_j = N$ and $m_j \geq 1$ while invertibility of $A(\bar{U}(0^\pm))$ by noncharacteristicity of $\bar{U}$
gives Lax's inequality
\[\lambda_{j^\prime}(\bar{U}(0^+)) < 0 < \lambda_{j^\prime}(\bar{U}(0^-)).\]

For some $T > 0$, suppose we are given a solution $\tilde{U}$ of \eqref{equ:general_balance_law} with subshock at $\eta$ such that for some (small) $\varepsilon > 0$
\begin{align}
    \label{est:smallness_shift_nonloc_dis}
    \max\limits_{t \in [0,T]} |\dot{\eta}(t)| \leq \varepsilon,
\end{align}
and the shifted perturbation variable
\[U(t,x) = \tilde{U}(t,x+\eta(t)) - \bar{U}(x)\]
satisfies
\begin{align}
        \label{est:smallness_perturbation_nonloc_dis}
        \max\limits_{t \in [0,T]} \norm{U(t)}{C^1_b(\dot{\RR})} \leq \varepsilon,
\end{align}
and
\begin{align}
    \label{est:boundedness_initial_perturbation_nonloc}
    \norm{U(0)}{C^{K-1}_b(\dot{\RR})} \leq \tilde{R}
\end{align}
In order to denote higher derivatives of $\eta$ more easily we will write
\[\partial_t \eta = \eta_t = \dot{\eta}, \partial_t^2 \eta = \eta_{tt} = \ddot{\eta}, \ldots\]
in the following.

The perturbation variable $U$ satisfies the interior equation
\begin{align}
    \label{U_equation_interior}
    U_t + \left(A(\tilde{U}) - \eta_t\right) U_x = & 
    Q(\tilde{U})U + \eta_t \bar{U}_x + M_1(\bar{U},\bar{U}_x,\tilde{U},U)
\end{align}
on $[0,T] \times \dot{\RR}$ with
\begin{align*}
    M_1(\bar{U},\bar{U}_x,\tilde{U},U) & = \left(A(\bar{U}) - A(\tilde{U})\right) \bar{U}_x + q(\tilde{U}) - q(\bar{U}) - Q(\tilde{U})U\\
    \nonumber & = \mathcal{O}(|\bar{U}_x||U|) + \mathcal{O}(|U|^2) = \mathcal{O}(|U|),
\end{align*}
and the boundary equation
\begin{align}
    \label{U_equation_boundary}
    \eta_t \left[\tilde{U}\right]_0 - \left[A(\tilde{U})U\right]_0 = & \left[N_1(\bar{U}, \tilde{U}, U)\right]_0
\end{align}
on $[0,T]$ with
\begin{align*}
    N_1(\bar{U}, \tilde{U}, U) & = F(\tilde{U}) - F(\bar{U}) - A(\tilde{U})U\\
    \nonumber& = \mathcal{O}(|U|^2) = \mathcal{O}(\varepsilon)\mathcal{O}(|U|).
\end{align*}

For convenience, set
\begin{align}
    \Upsilon^{(k)}(t,x) := L(\tilde{U}(t,x+\eta(t)))\partial_x^k U(t,x), \quad k = 0, 1, \ldots, K,
\end{align}
and
\begin{align*}
    \tilde{\Upsilon}^{(k)} := \Upsilon^{(k)} + \Theta \Upsilon^{(k-1)}, \quad k = 1, 2, \ldots, K,
\end{align*}
for $\Theta$ defined on $[0,T] \times \dot{\RR}$ as in \eqref{Theta_definition}. We keep the previous notation of components with $\Upsilon^{(k)}_j(t,x), \tilde{\Upsilon}^{(k)}_j(t,x) \in \RR^{m_j}$, and the symbols
\[\Phi = \Upsilon^{(0)}, \Psi = \Upsilon^{(1)}, \tilde{\Psi} = \tilde{\Upsilon}^{(1)}.\]

In the ``interior'' region $[0,T] \times \dot{\RR}$ we may proceed as before to show:

\begin{lem}
    \label{lemma_interior_damping_estimate}
    Let $j = 1,\ldots,J,$ and $k \in \lbrace 0, \ldots, K \rbrace$. Then there exist constants $\theta, C > 0$ such that the following holds: For $(t,x) \in [0,T] \times \dot{\RR}$, let $X_j$ define the
    $j$-th characteristic through $(t,x)$, that is $X_j$ solves \eqref{definition_j_X} on
    \[\overline{\lbrace s \in [0,t]:\; x\cdot X_j(s) > 0 \rbrace}.\]
    Set
    \begin{align*}
        \tilde{t} = \inf \lbrace s \in [0,t]:\; x\cdot X_j(s) > 0 \rbrace.
    \end{align*}
    Then the interior pointwise estimate
    \begin{align}
        \label{interior_damping_estimate}
        |\tilde{\Upsilon}^{(k)}_j(t,x)| \leq & C e^{-\theta(t-\tilde{t})} |\tilde{\Upsilon}^{(k)}_j(\tilde{t},X_j(\tilde{t}))|\\
        \nonumber & + C e^{-\theta t} \int_{\tilde{t}}^t e^{\theta s} \left(\sum_{k^\prime = 0}^{k-1}\norm{\Upsilon^{(k^\prime)}(s)}{C^0_b(\dot{\RR})} + |\dot{\eta}(s)|\right)ds\\
        \nonumber & + \varepsilon C e^{-\theta t} \int_{\tilde{t}}^t e^{\theta s} \norm{\tilde{\Upsilon}(s)}{C^0_b(\dot{\RR})}ds
    \end{align}
     holds.
\end{lem}
Fix a positive constant $\theta >0$ such that the conclusion of Lemma \ref{lemma_interior_damping_estimate} holds for some large constant $C>0$ that will possibly increase during the course of the following analysis.

We begin by demonstrating the estimate on $\Phi$ using techniques for noncharacteristic hyperbolic boundary value problems as in
\cite{K70, M83,M01} (see also \cite{YZ20}). Collect the $N+1$ ``incoming'' waves in
\begin{align}
    \label{def:incoming_waves}
    \Phi^{in}(t,x) := \begin{pmatrix}
    \Phi_{j^\prime}(t,-x) \\ \vdots \\ \Phi_J(t,-x) \\ \Phi_1(t,x) \\ \vdots \\\Phi_{j^\prime}(t,x)
    \end{pmatrix} \in \RR^{(N+1) \times 1}, \quad x \geq 0.
\end{align}
For each of the components $\Phi_j$ with $j^\prime\leq j \leq J$ if $x<0$, and $\Phi_j$ with $1 \leq j \leq j^\prime$ if $x>0$,
the interior estimate \eqref{interior_damping_estimate} holds with $\tilde{t} = 0$ as desired. Thus $\Phi^{in}$ satisfies the estimate
\begin{align}
    \label{def:incoming_damping_estimate}
    \norm{\Phi^{in}(t)}{C^0_b((0,\infty))} \leq & C e^{-\theta t} \norm{\Phi^{in}(0)}{C^0_b((0,\infty))} \\
      \nonumber & + C e^{-\theta t} \int_{0}^t e^{\theta s} \left(\norm{\Phi(s)}{C^0_b(\RR\backslash \lbrace 0 \rbrace)} + |\dot{\eta}(s)|\right)ds.
\end{align}
The characteristics associated with the $N-1$ ``outgoing'' waves, that is $\Phi_j$ with $1\leq j \leq j^-=j^\prime -1$ if $x<0$, and $\Phi_j$ with $j^+=j^\prime + 1 \leq j \leq J$ if $x >0$,
will however reach the ``jump'' boundary $\lbrace x = 0 \rbrace$ before reaching the ``initial'' boundary $\lbrace t = 0 \rbrace$ if they cross through $(t,x)$ with $|x|$ small enough.
Therefore, at the time $\tilde{t}$ when these characteristics reach the jump discontinuity we need an estimate on $\Phi(\tilde{t},0^\pm)$ in terms
of the incoming waves $\Phi^{in}$. The Lopatinski condition \eqref{equ:Lopatinski_condition} is used to deduce such an estimate.

Summarize the outgoing waves in the $N-1$ vector
\begin{align}
    \label{def:outgoing_waves}
    \Phi^{out}(t,x) = \begin{pmatrix} \Phi_1(t,-x) \\ \vdots \\ \Phi_{j^-}(t,-x) \\ \Phi_{j^+}(t,x) \\ \vdots \\ \Phi_J(t,x) \end{pmatrix} \in \RR^{(N-1) \times 1}, \quad x \geq 0,
\end{align}
and define
\begin{align*}
    \bar{R}^{in} & := \left(\lbrace R_j(\bar{U}(0^-)) \rbrace_{j^\prime \leq j \leq J},\lbrace R_j(\bar{U}(0^+)) \rbrace_{1 \leq j \leq j^\prime} \right) \in \RR^{N \times (N+1)},\\
    \bar{R}^{out} & := \left(\lbrace R_j(\bar{U}(0^-)) \rbrace_{1 \leq j \leq j^-},\lbrace R_j(\bar{U}(0^+)) \rbrace_{j^+ \leq j \leq J} \right) \in \RR^{N \times (N-1)}.
\end{align*}
Take orthogonal vectors $\bar{b}_1,\ldots, \bar{b}_{N-1} \in \RR^N$ such that $\bar{B}=(\bar{b}_1 \cdots \bar{b}_{N-1}) \in \RR^{N \times (N-1)}$
satisfies
\[\text{span}\, \bar{B} = (\text{span}\,\lbrace \left[\bar{U}\right]_0\rbrace)^\bot.\]
Then the Lopatinski condition \eqref{equ:Lopatinski_condition} implies
\begin{align}
    \label{Lopatinski_condition_orthogonal}
    \det(\bar{B}^T \bar{R}^{out}) \neq 0.
\end{align}
For $\varepsilon$ small enough, the Gram-Schmidt procedure applied to the basis
\[\left\lbrace \left[\tilde{U}(t,\cdot + \eta(t))\right]_0, b_1, \ldots, b_{N-1} \right\rbrace\] gives $N-1$ orthogonal vectors $\tilde{b}_1(t), \ldots, \tilde{b}_{N-1}(t)$ such that the matrix 
$\tilde{B}(t) = (\tilde{b}_1(t), \ldots, \tilde{b}_{N-1}(t))$ satisfies
\begin{align}
    \label{tilde_B_orthogonal_to_jump}
    \text{span}\, \tilde{B}(t) = (\text{span}\,\lbrace \left[\tilde{U}(t)\right]_0\rbrace)^\bot
\end{align}
and
\begin{align}
    \label{def:tilde_B}
    \tilde{B}(t) = \bar{B} + \mathcal{O}(\varepsilon), \quad t \in [0,T],
\end{align}
since $U(t,x) = \tilde{U}(t,x+\eta(t)) - \bar{U}(x) = \mathcal{O}(\varepsilon)$. Setting
\begin{align*}
    \tilde{R}^{in}(t) & := \left(\lbrace R_j(\tilde{U}(t,\eta(t)^-)) \rbrace_{j^\prime \leq j \leq J},\lbrace R_j(\tilde{U}(t,\eta(t)^+)) \rbrace_{1 \leq j \leq j^\prime} \right),\\
    \tilde{R}^{out}(t) & := \left(\lbrace R_j(\tilde{U}(t,\eta(t)^-)) \rbrace_{1 \leq j \leq j^-},\lbrace R_j(\tilde{U}(t,\eta(t)^+)) \rbrace_{j^+ \leq j \leq J} \right),
\end{align*}
and taking $\varepsilon$ small enough yields
\begin{align}
    \label{Lopatinski_condition_orthogonal_t}
    \tilde{B}(t)^T \tilde{R}^{out}(t) = \bar{B}^T\bar{R}^{out} + \mathcal{O}(\varepsilon) \neq 0
\end{align}
by \eqref{Lopatinski_condition_orthogonal}.

Using the $N \times N-1$ matrix $\tilde{B}(t)$ the $N$ boundary conditions \eqref{U_equation_boundary} are reduced
to (suppressing the $t$ evaluation in some places)
\begin{align}
    \label{reduced_boundary_conditions}
    - \tilde{B}(t)^T\left[R(\tilde{U})\Lambda(\tilde{U})\Phi\right]_0 & = \tilde{B}(t)^T \left(\eta_t \left[\tilde{U}\right]_0 - \left[A(\tilde{U})R(\tilde{U})\Phi\right]_0\right)\\
    \nonumber& = \mathcal{O}(\varepsilon)\mathcal{O}(|\Phi(t, 0^\pm)|) \in \RR^{N-1}.
\end{align}
A rearrangement of terms in \eqref{reduced_boundary_conditions} together with \eqref{tilde_B_orthogonal_to_jump} and \eqref{Lopatinski_condition_orthogonal_t} gives
\[\Phi^{out}(t,0^+) = (C^{out}(t))^{-1}C^{in}(t) \Phi^{in}(t,0^+) + \mathcal{O}(\varepsilon)\mathcal{O}(|\Phi(t, 0^\pm)|)\]
where
\[C^{in/out}(t) = \tilde{B}(t)^T \tilde{R}^{in/out}(t) \Lambda^{in/out}(t)\]
and
\begin{align*}
    \Lambda^{in}(t) & = \text{diag}\left(\lbrace \lambda_j(\tilde{U}(t,\eta(t)^-)) I_{m_j} \rbrace_{j^\prime \leq j \leq J},\lbrace -\lambda_j(\tilde{U}(t,\eta(t)^+)) I_{m_j} \rbrace_{1 \leq j \leq j^\prime} \right),\\
    \Lambda^{out}(t) & := \text{diag} \left(\lbrace -\lambda_j(\tilde{U}(t,\eta(t)^-)) I_{m_j} \rbrace_{1 \leq j \leq j^-},\lbrace \lambda_j(\tilde{U}(t,\eta(t)^+)) I_{m_j} \rbrace_{j^+ \leq j \leq J} \right),
\end{align*}
For small $\varepsilon$, we obtain the Kreiss-type relation
\begin{align}
    \label{Kreiss_type_estimate}
    |\Phi^{out}(t,0^+)| \leq C |\Phi^{in}(t,0^+)|
\end{align}
for some constant $C>0$.

The interior damping estimates in Lemma \ref{lemma_interior_damping_estimate} in
combination with the damping estimate for the incoming waves \eqref{def:incoming_damping_estimate} and
the Kreiss-type relation \eqref{Kreiss_type_estimate} prove the outgoing estimate
\begin{align*}
    \norm{\Phi^{out}(t)}{C^0_b((0,\infty))} \leq & C e^{-\theta t} \norm{\Phi(0)}{C^0_b(\RR\backslash \lbrace 0 \rbrace)} \\
      \nonumber & + C e^{-\theta t} \int_{0}^t e^{\theta s} \left(\norm{\Phi(s)}{C^0_b(\RR\backslash \lbrace 0 \rbrace)} + |\dot{\eta}(s)|\right)ds
\end{align*}
for all $t \in [0,T]$, and thus estimate \eqref{est:damping_nonloc_dis} for $K=0$.

Let us turn to the derivation of the damping estimates for $\Psi$. Because the components of $\tilde{\Psi} = \tilde{\Upsilon}^{(1)}$ satisfy
the interior estimates \eqref{interior_damping_estimate} it only remains to comment on the treatment of
the boundary terms at $x=0$ arising for the outgoing waves $\tilde{\Psi}^{out}$.

To this end, one differentiates the boundary equation \eqref{U_equation_boundary} with respect to time $t$,
obtaining boundary relations for  $x$-derivatives of $U$ by using the interior $U$-equation \eqref{U_equation_interior} solved for $U_t$.
Since $\bar{U}$ is independent of $t$, we can equivalently apply $\partial_t$ to the $\tilde{U}$ jump condition
\[{\eta}_t[\tilde{U}(t,\cdot + \eta(t))]_0 = [F(\tilde{U}(t,\cdot+\eta(t)))]_0,\]
and obtain
\begin{align}
    \label{t_derivative_of_boundary_equation}
    \eta_{tt} [\tilde{U}]_0 + \eta_t \left[U_t\right]_0 = \left[A(\tilde{U}) U_t\right]_0
\end{align}
where
\[U_t(t,x) = \partial_t(\tilde{U}(t,x+\eta(t)))\]
was used. Using \eqref{U_equation_interior} to replace $U_t$ in \eqref{t_derivative_of_boundary_equation} and rearranging the result gives
\begin{align}
    \label{tilde_Psi_boundary_equation}
    \eta_{tt} \left[\tilde{U}\right]_0 + \left[R(\tilde{U}) \Lambda(\tilde{U})^2 \tilde{\Psi}\right]_0 = & \mathcal{O}(|\eta_t|) + \mathcal{O}(\varepsilon)\mathcal{O}(|\tilde{\Psi}(t,0^\pm)|).\\
    \nonumber & + \mathcal{O}(|\Phi(t,0^\pm)|).
\end{align}
To be able to proceed as before one needs to control the $\mathcal{O}(|\eta_t|)$ in the above equation \eqref{tilde_Psi_boundary_equation}.

In the previous deduction
of the boundary estimate for $U$ we have not made use of the full boundary equation \eqref{U_equation_boundary} but only of the part orthogonal to $\lbrack \tilde{U}\rbrack_0$. The remaining part is utilized by multiplying 
\eqref{U_equation_boundary} by $\lbrack \tilde{U}\rbrack_0^T$ which yields
\begin{align*}
    \eta_t = \frac{\lbrack \tilde{U}\rbrack_0^T \left([A(\tilde{U}) U + N_1(\bar{U},\tilde{U},U)]_0\right)}{\norm{\lbrack \tilde{U}\rbrack_0}{}^2} = \mathcal{O}(|\Phi(t,0^\pm)|),
\end{align*}
and thus
\begin{align}
    \label{tilde_Psi_boundary_equation_for_estimate}
    \eta_{tt} \left[\tilde{U}\right]_0 + \left[R(\tilde{U}) \Lambda(\tilde{U})^2 \tilde{\Psi}\right]_0 = \mathcal{O}(\varepsilon)\mathcal{O}(|\tilde{\Psi}(t,0^\pm)|) + \mathcal{O}(|\Phi(t,0^\pm)|).
\end{align}
Multiplication of \eqref{tilde_Psi_boundary_equation_for_estimate} by $\tilde{B}(t)^T$ and sorting for incoming and outgoing waves gives
\[|\tilde{\Psi}^{out}(t,0^+)| \leq C (|\tilde{\Psi}^{in}(t,0^+)| + |\Phi(t,0^\pm)|).\]
Making use of these preparations and the already shown damping estimate for $\Phi$ while also eliminating small cross-terms using Gronwall's Lemma
as before yields the damping estimate
\begin{align*}
    \norm{\Psi(t)}{C^0_b} \leq & C e^{-\frac{\theta}{2} t} \norm{(\Phi(0),\Psi(0))}{C^0_b}\\
    \nonumber & + C \int_0^t e^{-\frac{\theta}{2}(t - s)} \left(\norm{\Phi(s)}{C^0_b} + |\dot{\eta}(s)|\right)ds.
\end{align*}

To deal with higher derivatives like
\[\Upsilon = \Upsilon^{(2)} = L(\tilde{U}(t,x+\eta(t))) \partial^2_x U(t,x)\]
one continues to take derivatives of the boundary equation \eqref{U_equation_boundary} with respect to time $t$, and eliminates the
additional $\Psi$-forcing as before.

\begin{rem}
    Again, the smallness of derivatives of $U$ of order higher than $1$ is not needed since the highest order derivatives of $U$ that appear in the estimates can only have factors involving contributions from $U, U_x, \bar{U}, \bar{U}_x,$ and $\dot{\eta}$. Nonetheless,
    with each step towards higher derivatives one has to enlarge the constant $C$ and shrink the exponent to $\theta/(2K) > 0$ in the damping estimate due to additional lower-order
    and forcing terms.
\end{rem}
\end{proof}
\section{Localized damping estimates}
\label{sec:localized_damping_estimates}

In this section we generalize the localized damping estimates of \cite{MZ02,MZ05} and \cite{YZ20} for smooth and discontinuous profiles to the general non-symmetric case. As stressed in the introduction,
our approach greatly simplifies the deduction of damping estimates. The previously employed delicate Kawashima-type estimates  are replaced by a simple variable transform like the one used in the nonlocalized
case (see \eqref{def:tilde_Psi}) in order to control large off-diagonal terms. Damping estimates are stated in terms of $L^2$-based Sobolev norms
\[\norm{f}{H^K(\Omega)} := \left(\sum_{k=0}^K \norm{\partial_x^k f}{L^2(\Omega)}^2\right)^{1/2}\]
where $\Omega \subset \RR$ is open and $f:\Omega \rightarrow \RR^M$ is a $K$-times differentiable function with integers $K = 0, 1, 2, \ldots, M = 1, 2, 3, \ldots$.

\subsection{The case of smooth profiles}
\label{subsec:localized_damping_estimates_smooth}
We begin with the damping estimates for smooth profiles. The statement of Theorem \ref{thm:damping_estimate_localized_smooth} takes smooth and compactly supported perturbations into account. For a local existence
theory of perturbations with values in Sobolev spaces the reader is referred to, e.g., \cite{Taylor_91}.
\begin{thm}
    \label{thm:damping_estimate_localized_smooth}
    Let $K \in \NN$ with $K \geq 2, \tilde{R}$, and let $\bar{U}$ be a noncharacteristic smooth profile of \eqref{equ:general_balance_law} with end states $U^-, U^+$ which are stable at high frequencies.
    Then there exist constants $C, \varepsilon, \theta > 0$ such that the following holds:
    
    Let $T>0$. If $\tilde{U}$ is a smooth solution of
    \eqref{equ:general_balance_law} on $[0,T] \times \RR$ such that the shifted perturbation variable
    \[U(t,x) = \tilde{U}(t,x+\eta(t)) - \bar{U}(x)\]
    is compactly supported and satisfies
    \begin{align*}
        \max\limits_{t \in [0,T]} \norm{U(t)}{H^2(\RR)} \leq \varepsilon
    \end{align*}
    and
    \begin{align*}
        \norm{U(0)}{H^K(\RR)} \leq \tilde{R}
    \end{align*}
    for a smooth shift $\eta: [0,T] \rightarrow \RR$ with
    \begin{align*}
        \max\limits_{t \in [0,T]} |\dot{\eta}(t)| \leq \varepsilon,
    \end{align*}
    then the damping estimate
    \begin{align}
        \label{est:damping_loc_smo}
        \norm{U(t)}{H^K(\RR)}^2 \leq C e^{-\theta t} \norm{U(0)}{H^K(\RR)}^2 + C \int_0^t e^{-\theta (t-s)}(\norm{U(s)}{L^2(\RR)}^2 + |\dot{\eta}(s)|^2) ds
    \end{align}
    holds for all $t \in [0,T]$.
\end{thm}
\begin{rem}
    \label{rem:independent_localized_statement}
    To obtain a version of Theorem \ref{thm:damping_estimate_localized_smooth} that is independent of the $L^\infty$-bounds implied by the previously obtained damping
    estimates in Theorem \ref{thm:damping_estimate_nonlocalized_smooth} the bound on the initial data is replaced by a bound on the solution on the entire interval $[0,T]$,
    \[\norm{U(t)}{H^K(\RR)} \leq \tilde{R}, \quad t \in [0,T].\]
\end{rem}
\begin{proof}
Let $K \in \NN$ be an integer with $K \geq 2$, and let $\bar{U}$ be a noncharacteristic smooth profile of \eqref{equ:general_balance_law} with end states $U^-, U^+$ which are stable at high frequencies.
By Definitions \ref{def:smooth_profile}, \ref{def:high_frequency_stability} and \ref{def:noncharacteristic} this means that there exist constants $\theta_k, c_k, c_{nc}, c_E$
such that
\begin{align}
    \label{est:profile_decay_0_loc_smo}
    |\partial_x^k(U(x) - U^\pm)|  & \leq c_k e^{-\theta_k |x|}, \quad x \gtrless 0, k \in \NN,\\
    \label{est:noncharacteristic_loc_smo}
    |\lambda_j(\bar{U}(x))| & \geq c_{nc} > 0, \quad x \in \RR, j \in \lbrace 1, \ldots, J \rbrace,\\
    \label{est:E_j_negative_loc_smo}
    \RE \sigma(E_j^\pm) & \leq - c_E < 0, \quad j \in \lbrace 1, \ldots, J \rbrace,
\end{align}
with $E_j^\pm = E_j(U^\pm)$ and $E_j$ as defined in \eqref{def:E_diag}.

For some time $T>0$, suppose we are given a smooth solution $\tilde{U} \in C^\infty_0([0,T] \times \RR)$ of
the system of balance laws \eqref{equ:general_balance_law}, and a smooth phase shift $\eta: [0,T] \rightarrow \RR$ such that for some (small) constant $\varepsilon> 0$ 
\[|\dot{\eta}(t)| \leq \varepsilon, \quad t \in [0,T],\]
and the shifted perturbation variable
\[U(t,x) = \tilde{U}(t,x+\eta(t)) - \bar{U}(x)\]
is compactly supported and satisfies
\begin{align}
    \label{est:U_smallness_loc_smo}
    \norm{U(t)}{H^{2}(\RR)} \leq \varepsilon, \quad t \in [0,T],
\end{align}
and
\begin{align}
    \label{est:boundedness_initial_data_loc_smo}
    \norm{U(0)}{H^K(\RR)} \leq \tilde{R}
\end{align}
Notice that by Sobolev's inequality, \eqref{est:U_smallness_loc_smo} and \eqref{est:boundedness_initial_data_loc_smo} imply
\begin{align}
    \label{est:U_smallness_and_boundadness_loc_smo_infty}
    \norm{U(t)}{W^{1, \infty}(\RR)} = \mathcal{O}(\varepsilon), t \in [0,T], \textnormal{ and } \norm{U(0)}{W^{K-1, \infty}(\RR)} = \mathcal{O}(\tilde{R}).
\end{align}
In the following, we will usually drop the reference to $\RR$ in all $L^2$-based norms appearing in the proof.

Define $\Phi, \Psi$ and $\Upsilon$ as before. An estimate for $\Phi$ of the form
\begin{align}
        \label{Phi_L_2_damping_estimate}
        \norm{\Phi(t)}{L^2}^2 \leq C e^{-\theta t} \norm{\Phi(0)}{L^2}^2  + C \int_0^t e^{-\theta (t-s)} \left(\norm{\Phi(s)}{L^2}^2 + |\dot{\eta}(s)|^2 \right) ds
\end{align}
with arbitrary $\theta > 0$ and some $C=C(\theta, \varepsilon) > 0$ is easily obtained by adding and subtracting a large multiple of $\Phi$ in the equation satisfied by $\Phi$,
\begin{align}
        \label{Phi_equation_rearranged}
        \Phi_t + \Lambda(\bar{U}) \Phi_x =
        & \mathcal{O}(|\Phi|) + \dot{\eta}\mathcal{O}(e^{\mathcal{O}(-|x|)}) + \dot{\eta} \Phi_x + (\Lambda(\bar{U}) - \Lambda(\tilde{U})) \Phi_x
\end{align}
where
\[|\partial_x^k (\Lambda(\bar{U}) - \Lambda(\tilde{U}))| | = \mathcal{O}(\varepsilon), \quad k =0,1,\]
and applying standard energy methods such as partial integration.
\begin{rem}
    Notice that in \eqref{Phi_equation_rearranged} the $\dot{\eta}$-term pairs with an exponentially decaying term coming from the derivative of the profile $\bar{U}$ as one
    can see in \eqref{U_equation}. So far we have only used the boundedness of the exponentially decaying contribution but in the localized setting integrability is needed, 
    and thus the more precise expression $\dot{\eta}\mathcal{O}(e^{\mathcal{O}(-|x|)})$ is used.

    Also, the $\mathcal{O}(|\Phi|)$-term in \eqref{Phi_equation_rearranged} contains products of (components of) $\Phi$ and $\Psi$ (but also other terms involving $\Phi$). Sobolev's inequality is used, as in \eqref{est:U_smallness_and_boundadness_loc_smo_infty},
    to produce the relevant estimates indicated by the $\mathcal{O}$-notation. It is not of the form $\mathcal{O}(\varepsilon|\Phi|)$ since also products of $\Phi$ and derivatives of $\bar{U}$ 
    appear.
\end{rem}

Arranging the equation for $\Psi$ in a similar way as done for $\Phi$ in \eqref{Phi_equation_rearranged}
one obtains
\begin{align}
    \label{equ:Psi_equation_rearranged}
    \Psi_t + \Lambda(\bar{U}) \Psi_x = & \lbrace E(\bar{U}) + F(\bar{U}) \rbrace \Psi + \mathcal{O}(|\bar{U}_x|)\Psi\\
    \nonumber & + \mathcal{O}(\varepsilon) \Psi + \mathcal{O}(|\Phi|) + \dot{\eta}\mathcal{O}\left(e^{\mathcal{O}(-|x|)}\right)\\
    \nonumber & + \dot{\eta} \Psi_x + (\Lambda(\bar{U}) - \Lambda(\tilde{U})) \Psi_x
\end{align}
where $\Lambda$, like $E$ and $F$, is being evaluated at the profile $\bar{U}$ instead of $\tilde{U}$ like in \eqref{Psi_equation} at the cost of two terms involving $\Psi_x$ which can be treated by partial integration and smallness of $\varepsilon$.

\begin{rem}
    For the terms $\mathcal{O}(\varepsilon) \Psi$ and $\mathcal{O}(|\Phi|)$ in \eqref{equ:Psi_equation_rearranged} the Sobolev inequality is invoked again to estimate
    absolute values of $\Phi$ and $\Psi$. With the exception of Remark \ref{rem:smallness_order_loc} we will not comment on similar arguments in the following.
\end{rem}

Let $0 < \theta < c_E$. To make use of the high-frequency damping as given in \eqref{est:E_j_negative_loc_smo} suitable bases have to be chosen at $\pm \infty$ (compare with the proof of
Lemma \ref{lem:damping_estimate}). By \eqref{est:E_j_negative_loc_smo}, we find invertible matrices $M^\pm_j \in \RR^{m_j \times m_j}, j = 1, \ldots, J,$ with 
\begin{align*}
    \det M^\pm_j > 0
\end{align*}
such that the real, or in our case symmetric, part of
$\tilde{E}^\pm_j := (M^\pm)^{-1}_j E^\pm_j M^\pm_j$ is negative definite with
\begin{align*}
    \RE \tilde{E}^\pm_j \leq -\theta I_{m_j} < 0.
\end{align*}
Since the set $\textnormal{GL}(\RR, m_j)$ of invertible $m_j \times m_j$ matrices with real coefficients consists of the two pathwise connected open components $\det^{-1}((-\infty,0))$ and $\det^{-1}((0,\infty))$
there exists a smooth function
\[M_j: [-1,1] \rightarrow \textnormal{GL}(\RR, m_j) \textnormal{ with } M_j(-1) = M^-_j, M_j(+1) = M^+_j.\]
Setting
\[M(x) = \diag (M_1(\tanh(x)), \ldots, M_J(\tanh(x)))\]
for $x \in \RR$, and
\[M^\pm = \diag (M_1^\pm, \ldots, M_J^\pm),\]
we find constants $C_{M,k}, \theta_{M,k} > 0$ with $k \in \NN$ such that
\begin{align}
    \label{est:M_decay}
    |\partial_x^k (M(x) - M^\pm)| \leq C_{M,k} e^{-\theta_{M,k}|x|}, \quad x \gtrless 0, k \in \NN.
\end{align}
Set
\[E_M(x) := M^{-1}(x) E(\bar{U}(x)) M(x), \quad F_M(x) := M^{-1}(x) F(\bar{U}(x)) M(x),\]
and notice that due to the block-structures
\begin{align}
    \label{equ:E_M_F_M_properties}
    \begin{aligned}
    (E_M)_{jk} & = 0 \textnormal{ if } j \neq k,\\
    (F_M)_{jj} & = 0,
    \end{aligned}
\end{align}
for $j,k = 1, \ldots, J$, and
\begin{align}
    \label{equ:Lambda_M_commute}
    M^{-1} \Lambda(\bar{U}) = \Lambda(\bar{U}) M^{-1}.
\end{align}
Furthermore, by construction the limiting matrices $E^\pm_M = \lim_{x \rightarrow \pm \infty} E_M(x)$ have negative definite symmetric part with
\begin{align}
    \label{ieq:E_M_negative_definite}
    \RE E_M^\pm \leq -\theta I_N,
\end{align}
and $E_M$ approaches $E^\pm_M$ at an exponential rate as $x \rightarrow \pm \infty$ by \eqref{est:profile_decay_0_loc_smo} and \eqref{est:M_decay}.

Again, we make a change of variables
\begin{align}
    \label{def:Psi_tilde_loc}
    \tilde{\Psi} := M^{-1} \Psi + \Theta M^{-1} \Phi
\end{align}
using $M$ to reveal the damping at $\pm \infty$, and a new $\Theta$ which will eliminate large cross-terms in \eqref{equ:Psi_equation_rearranged}. In the present localized setting we may arrange for $\Theta$ to be
bounded and also independent of time $t$.
The equation for $\tilde{\Psi}$ has the form
\begin{align}
    \label{equ:Psi_tilde_equation}
    \tilde{\Psi}_t + \Lambda(\bar{U}) \tilde{\Psi}_x = & E_M(x) \tilde{\Psi} + \lbrace \Lambda(\bar{U}) \Theta - \Theta \Lambda(\bar{U}) + F_M(x)\rbrace \tilde{\Psi}\\
    \nonumber & + \mathcal{O}(|\Phi|) + \dot{\eta}\mathcal{O}\left( e^{\mathcal{O}(-|x|)} \right) + \mathcal{O}(\varepsilon) \tilde{\Psi} + \mathcal{O}\left( e^{\mathcal{O}(-|x|)} \right) \tilde{\Psi}\\
    \nonumber & + \dot{\eta} \tilde{\Psi}_x + (\Lambda(\bar{U}) - \Lambda(\tilde{U})) \tilde{\Psi}_x
\end{align}
where \eqref{equ:Lambda_M_commute} was used, too.
Define for $j,k = 1, \ldots, J$ 
\begin{align}
    \label{def:Theta_loc}
    \Theta(x)_{jk} := \begin{cases}
        \frac{F_M(x)}{\lambda_k(\bar{U}(x)) - \lambda_j(\bar{U}(x))} & \textnormal{if } j \neq k,\\
        0 & \textnormal{if } j = k.
    \end{cases}
\end{align}
Then large, non-decaying cross-terms are eliminated in \eqref{equ:Psi_tilde_equation} (note the second line in \eqref{equ:E_M_F_M_properties}), and we obtain
\begin{align}
    \label{equ:Psi_tilde_equation_final}
    \tilde{\Psi}_t + \Lambda(\bar{U}) \tilde{\Psi}_x = & E_M(x) \tilde{\Psi} + \mathcal{O}\left( e^{\mathcal{O}(-|x|)} \right) \tilde{\Psi} \\
    \nonumber & + \mathcal{O}(\varepsilon) \tilde{\Psi} + \dot{\eta} \tilde{\Psi}_x + (\Lambda(\bar{U}) - \Lambda(\tilde{U})) \tilde{\Psi}_x\\
    \nonumber & + \mathcal{O}(|\Phi|) + \dot{\eta}\mathcal{O}\left(e^{\mathcal{O}(-|x|)}\right)
\end{align}

In order to deal with terms that have exponentially decaying coefficients in \eqref{equ:Psi_tilde_equation_final}, a weighted energy method is applied as in \cite{MZ05}.
For $j = 1, \ldots, J$, the Goodman-type weights $\alpha_j$ \cite{G91} are defined as solutions 
of the scalar ordinary differential equations
\begin{align}
    \label{definition_alpha_j}
    \partial_x \alpha_j = - \frac{C_\alpha e^{-c_\alpha |x|}}
    {\lambda_j(\bar{U}(x))} \alpha_j(x),
    \end{align}
where the constants $c_\alpha, C_\alpha > 0$ are to be determined later. The weights $\alpha_j$ are well-defined by noncharacteristicity \eqref{est:noncharacteristic_loc_smo}.
No matter the choice of these constants, one can normalize each $\alpha_j$ such that
\begin{align}
        \label{alpha_j_bounds}
        0 < \inf\limits_{x \in \RR} \alpha_j(x) < \sup\limits_{x \in \RR} \alpha_j(x) = 1.
\end{align}
Collect all weights in a diagonal matrix
\begin{align}
    \label{def:weighting_matrix}
    \Omega(x) := \diag(\alpha_1(x) I_{m_1}, \ldots, \alpha_J(x) I_{m_J}).
\end{align}
The equations \eqref{definition_alpha_j} and the bounds \eqref{alpha_j_bounds} imply that the derivative of $\Omega$ satisfies
\begin{align}
    \label{equ:Omega_prime}
    \Omega^\prime(x) = C_\alpha e^{-\alpha|x|} \Omega(x) \Lambda(\bar{U}(x))^{-1},
\end{align}
and the weighted norm
\begin{align*}
    \norm{\tilde{\Psi}}{L^2_\Omega(\RR)} := \sqrt{\skp{\tilde{\Psi}, \Omega \tilde{\Psi}}{L^2(\RR)}}
\end{align*}
is equivalent to the usual $L^2$-norm. Taking the time derivative of the weighted $L^2$-inner product $\skp{\tilde{\Psi}(t), \Omega \tilde{\Psi}(t)}{L^2(\RR)}$ and
partial integration yield for some positive constants $C, c >0$
\begin{align*}
    \frac{d}{dt}\norm{\tilde{\Psi}}{L^2_\Omega}^2 = & \skp{\tilde{\Psi}_t, \Omega(x) \tilde{\Psi}}{L^2} + \skp{\tilde{\Psi}, \Omega(x) \tilde{\Psi}_t}{L^2}\\
    \leq & \skp{\tilde{\Psi}, \Omega^\prime(x) \Lambda(\bar{U}(x)) \tilde{\Psi}}{L^2} - \frac{3\theta}{2} \norm{\tilde{\Psi}}{L^2_\Omega}^2\\
    & + \skp{\tilde{\Psi}, \Omega(x) Ce^{-c|x|} \tilde{\Psi}}{L^2} + \skp{\tilde{\Psi}, C \varepsilon \Omega^\prime(x) \tilde{\Psi}}{L^2}\\
    & + C \varepsilon \norm{\tilde{\Psi}}{L^2_\Omega}^2 + C (\norm{\Phi}{L^2}^2 + |\dot{\eta}|^2)
\end{align*}
where \eqref{ieq:E_M_negative_definite} and Young's inequality were used, too. Notice that $\Omega(x)$ commutes with $\Lambda(\bar{U}(x))$ and $E_M(x)$. The equation \eqref{equ:Omega_prime} for $\Omega^\prime$ and noncharacteristicity \eqref{est:noncharacteristic_loc_smo} give
\begin{align*}
    \frac{d}{dt}\norm{\tilde{\Psi}}{L^2_\Omega}^2 \leq & - \frac{3 \theta}{2} \norm{\tilde{\Psi}}{L^2_\Omega}^2 + C \varepsilon \norm{\tilde{\Psi}}{L^2_\Omega}^2 + C (\norm{\Phi}{L^2}^2 + |\dot{\eta}|^2)\\
    & + \skp{\tilde{\Psi}, \Omega(x) \left(Ce^{-c|x|} + \left(\varepsilon C c_{nc}^{-1} - \frac{1}{2}\right) C_\alpha e^{-c_\alpha|x|} \right) \tilde{\Psi}}{L^2}.
\end{align*}
By choosing
\[c_\alpha = c, C_\alpha = 8C, \textnormal{ and } \varepsilon \leq \frac{1}{4C}\min \lbrace \theta, c_{nc} \rbrace,\]
one obtains
\begin{align*}
    \frac{d}{dt}\norm{\tilde{\Psi}}{L^2_\Omega}^2 \leq & - \theta\norm{\tilde{\Psi}}{L^2_\Omega}^2 + C (\norm{\Phi}{L^2}^2 + |\dot{\eta}|^2)
\end{align*}
which implies the estimate
\begin{align}
    \label{Psi_L_2_damping_estimate_smo}
    \norm{\Psi(t)}{L^2}^2 \leq Ce^{-\theta t} \norm{(\Phi(0),\Psi(0))}{L^2}^2 + C \int_0^t e^{- \theta(t-s)}\left(\norm{\Phi(s)}{L^2}^2 + |\dot{\eta}(s)|^2 \right)ds
\end{align}
with a possibly large $C>0$.

Finally,
\[\tilde{\Upsilon} = M^{-1}\Upsilon + \Theta M^{-1}\Psi\]
satisfies an equation of the form
\begin{align*}
    \tilde{\Upsilon}_t + \Lambda(\bar{U}) \tilde{\Upsilon}_x = & E_M(x) \tilde{\Upsilon} + \mathcal{O}\left( e^{\mathcal{O}(-|x|)} \right) \tilde{\Upsilon} \\
    \nonumber & + \mathcal{O}(\varepsilon) \tilde{\Upsilon} + \dot{\eta} \tilde{\Upsilon}_x + (\Lambda(\bar{U}) - \Lambda(\tilde{U})) \tilde{\Upsilon}_x\\
    \nonumber & + \mathcal{O}(|\Phi|) + \mathcal{O}(|\Psi|) + \dot{\eta}\mathcal{O}\left(e^{\mathcal{O}(-|x|)}\right)
\end{align*}
which gives with the techniques described above an estimate of the form
\begin{align}
    \label{Upsilon_L_2_estimate_pre_final}
    \norm{\tilde{\Upsilon}(t)}{L^2}^2 \leq &C e^{-\theta t} \norm{\tilde{\Upsilon}(0)}{L^2}^2\\
    \nonumber & + C \int_0^t e^{-\theta(t-s)} \left(\norm{\Psi(s)}{L^2}^2 + \norm{\Phi(s)}{L^2}^2 +
    |\dot{\eta}(s)|^2\right)ds.
\end{align}
Similar to the $C^0_b$-case, use a large multiple of the $\Psi$ $L^2$-estimate \eqref{Psi_L_2_damping_estimate_smo}
to remove the $\Psi$ forcing in \eqref{Upsilon_L_2_estimate_pre_final}, and 
unwrap definitions to find the damping estimate
\begin{align}
    \label{Upsilon_L_2_estimate}
    \norm{\Upsilon(t)}{L^2}^2 \leq &C e^{-\frac{\theta}{2} t} \norm{(\Phi(0), \Psi(0), \Upsilon(0))}{L^2}^2\\
    \nonumber & + C \int_0^t e^{-\frac{\theta}{2}(t-s)} \left(\norm{\Phi(s)}{L^2}^2 +
    |\dot{\eta}(s)|^2\right)ds
\end{align}
which together with the estimates \eqref{Phi_L_2_damping_estimate}, \eqref{Psi_L_2_damping_estimate_smo}
implies \eqref{est:damping_loc_smo} for $K=2$ since $L(\tilde{U})$ and $R(\tilde{U})$ are bounded by \eqref{est:profile_decay_0_loc_smo} and \eqref{est:U_smallness_and_boundadness_loc_smo_infty}.
Estimates on higher derivatives can be obtained without further difficulties (but see the following remark).
\begin{rem}
    \label{rem:smallness_order_loc}
    In the derivation of the damping estimate \eqref{Upsilon_L_2_estimate} on $\Upsilon$ we make use of $|\Phi(t,x)| + |\Psi(t,x)| = \mathcal{O}(\varepsilon)$ by
    \eqref{est:U_smallness_and_boundadness_loc_smo_infty}. For $K \geq 3$, among other terms, integrals of products of the form
    \begin{align}
        \label{def:product_of_three}
        \partial_x^{k_1} U_j \partial_x^{k_2} U_l \partial_x^K U_m, \quad j,l,m \in \lbrace 1, \ldots, N \rbrace,
    \end{align}    
    with $2 \leq k_1, k_2 \leq K, k_1 + k_2 = K+1,$ have to be estimated when applying the (weighted) energy method described above (for simplicity we use $\partial^k_x U$ instead of
    the transformed variables $\Upsilon^{(k)}, \tilde{\Upsilon}^{(k)}$). By the $W^{K-1,\infty}$-bound in \eqref{est:U_smallness_and_boundadness_loc_smo_infty} on the initial data, we can invoke the
    result obtained in the nonlocalized case, Theorem \eqref{thm:damping_estimate_nonlocalized_smooth}, for derivatives of $U$ of order $K-1$ and smaller. Thus, an $L^\infty$-bound is available
    for $\partial_x^k U(t)$ if $k \leq K-1$. In combination with Young's inequality terms similar to \eqref{def:product_of_three} are treated. This avoids demanding smallness of $U(t)$ in a Sobolev norm involving derivatives higher than $2$.
\end{rem}
\end{proof}
\subsection{The case of discontinuous profiles}
\label{subsec:localized_damping_estimates_dis}
Building on the procedures used in the smooth profile case we address the localized damping estimates for discontinuous profiles. Like in \cite{YZ20}, the proof of Theorem \ref{thm:damping_estimate_localized_dis} makes
use of the ``maximally dissipative boundary conditions'' implied by the Lopatinski condition \ref{def:Lopatinski} in the case of one-dimensional space. Again, the theorem is stated for compactly supported perturbations
that are smooth in the interior. For a local existence theory for solutions with a single subshock and values in appropriate Sobolev spaces the reader is referred to the works of Majda \cite{M83} and Mètivier \cite{M01} (see also \cite{YZ20} for a suitable statement).
\begin{thm}
    \label{thm:damping_estimate_localized_dis}
    Let $K \in \NN$ with $K \geq 2, \tilde{R} > 0$, and let $\bar{U}$ be a noncharacteristic discontinuous profile of \eqref{equ:general_balance_law} with end states $U^-, U^+$
    which are stable at high frequencies. Assume that $(\bar{U}(0^+), \bar{U}(0^-))$ satisfies the Lopatinski condition.
    Then there exist constants $C, \varepsilon, \theta > 0$ such that the following holds:
    
    Let $T>0$. If $\tilde{U}$ is a solution of
    \eqref{equ:general_balance_law} on $[0,T]$ with subshock at $\eta$ such that
    \begin{align*}
        \max\limits_{t \in [0,T]} |\dot{\eta}(t)| \leq \varepsilon,
    \end{align*}
    and the shifted perturbation variable
    \[U(t,x) = \tilde{U}(t,x+\eta(t)) - \bar{U}(x)\]
    is compactly supported and satisfies
    \begin{align*}
        \max\limits_{t \in [0,T]} \norm{U(t)}{H^{2}(\dot{\RR})} \leq \varepsilon,
    \end{align*}
    and
    \begin{align*}
        \norm{U(0)}{H^K(\dot{\RR})} \leq \tilde{R}
    \end{align*}
    then the damping estimate
    \begin{align}
        \label{est:damping_loc_dis}
        \norm{U(t)}{H^K(\dot{\RR})}^2 \leq C e^{-\theta t} \norm{U(0)}{H^K(\dot{\RR})}^2 + C \int_0^t e^{-\theta (t-s)}(\norm{U(s)}{L^2(\dot{\RR})}^2 + |\dot{\eta}(s)|^2) ds
    \end{align}
    holds for all $t \in [0,T]$.
\end{thm}
\begin{proof}
    Let $K \in \NN$ with $K \geq 2, \tilde{R} > 0$, and let $\bar{U}$ be a noncharacteristic discontinuous profile of \eqref{equ:general_balance_law} with end states $U^-, U^+$
    which are stable at high frequencies. Assume that $(\bar{U}(0^+), \bar{U}(0^-))$ satisfies the Lopatinski condition as in Definition \ref{def:Lopatinski}.

    For some $T > 0$, suppose we are given a solution $\tilde{U}$ of \eqref{equ:general_balance_law} with subshock at $\eta$ such that for some (small) $\varepsilon > 0$
    \begin{align}
        \label{est:smallness_shift_loc_dis}
        |\dot{\eta}(t)| \leq \varepsilon, \quad t \in [0,T],
    \end{align}
    and the shifted perturbation variable
    \[U(t,x) = \tilde{U}(t,x+\eta(t)) - \bar{U}(x)\]
    is compactly supported and satisfies
    \begin{align}
        \label{est:smallness_perturbation_loc_dis}
        \norm{U(t)}{H^2(\dot{\RR})} \leq \varepsilon, \quad t \in [0,T],
    \end{align}
    and
    \begin{align}
        \label{est:boundedness_initial_data_loc}
        \norm{U(0)}{H^K(\dot{\RR})} \leq \tilde{R}.
    \end{align}
    
    Since the end states $U^\pm$ are stable at high frequencies there exist matrices $M^\pm_j \in \RR^{m_j \times m_j}, j \in \lbrace 1, \ldots, J \rbrace,$ and a constant $\theta > 0$ such that the matrices
    \[E_M(x)^\pm := (M^\pm)^{-1} E(\bar{U}(x)) M^\pm, \quad x \gtrless 0,\]
    converge to limiting matrices $E_M^{\pm}$ at exponential rate as $x \rightarrow \pm \infty$ whose symmetric parts are negative definite with
    \begin{align}
        \label{est:E_pm_negative_definite_loc_dis}
        \RE E_M^\pm \leq - \theta I_N.
    \end{align}
    Similarly, define
    \[F_M(x) := (M^\pm)^{-1} F(\bar{U}(x)) M^\pm, \quad x \gtrless 0,\]
    and note that, like in the previous proof, we have for $j,k \in \lbrace 1, \ldots, J \rbrace$
    \[(E_M(x))_{jk} = 0 \textnormal{ if } j \neq k, (F_M(x))_{jj} = 0, \textnormal{ as well as } [\Lambda(\bar{U}(x)), M^\pm] = 0.\]
    For the purpose of eliminating off-diagonal terms in the investigated equations, define for $j,k = 1, \ldots, J$ and $x \in \dot{\RR}$
    \begin{align*}
        \Theta(x)_{jk} := \begin{cases}
            \frac{F_M(x)}{\lambda_k(\bar{U}(x)) - \lambda_j(\bar{U}(x))} & \textnormal{if } j \neq k,\\
            0 & \textnormal{if } j = k,
        \end{cases}
    \end{align*}
    with appropriate limits $\Theta(0^\pm)$ as $x\rightarrow 0^\pm$.

    For $k \in \lbrace 0, \ldots, K \rbrace$ define the transformed variables
    \[\Upsilon^{(k)}(t,x) := L(\tilde{U}(t,x+\eta)) \partial^k_x U(t,x), \quad (t,x) \in [0,T] \times \dot{\RR}\]
    and
    \[\tilde{\Upsilon}^{(0)} := \Upsilon^{(0)}\]
    as well as 
    \[\tilde{\Upsilon}^{(k)}(t,x) := (M^\pm)^{-1} \Upsilon^{(k)}(t,x) + \Theta(x) (M^\pm)^{-1} \Upsilon^{(k-1)}(t,x), \quad t \in [0,T], x \gtrless 0,\]
    if $k>0$, all possessing appropriate limits $\Upsilon^{(k)}(t,0^\pm), \tilde{\Upsilon}^{(k)}(t,0^\pm)$ as $x \rightarrow 0^\pm$.
    
    For $k \leq K$, the ``interior'' equations satisfied by $\tilde{\Upsilon}^{(k)}$ on $[0,T] \times \dot{\RR}$ read
    \begin{align}
        \label{equ:interior_loc_dis}
        \tilde{\Upsilon}^{(k)}_t + \Lambda(\bar{U})\tilde{\Upsilon}^{(k)} = & E_M(x) \tilde{\Upsilon}^{(k)} + \mathcal{O}\left(e^{\mathcal{O}(-|x|)}\right)\tilde{\Upsilon}^{(k)} + \mathcal{O}(\varepsilon)\tilde{\Upsilon}^{(k)}\\
        \nonumber& + \dot{\eta} \tilde{\Upsilon}^{(k)}_x + (\Lambda(\bar{U})-\Lambda(\tilde{U}))\tilde{\Upsilon}^{(k)}_x\\
        \nonumber& + \dot{\eta} \mathcal{O}\left(e^{\mathcal{O}(-|x|)}\right) + \mathcal{O}\left(\sum_{k^\prime = 0}^{k-1} |\Upsilon^{(k^\prime)}|\right)
    \end{align}
    Equations \eqref{equ:interior_loc_dis} are augmented by the jump conditions, or boundary equations,
    \begin{align}
        \label{equ:jump_loc_dis}
        \left(\partial_t^{k+1}\eta\right) [\tilde{U}]_0 + \left[R(\tilde{U}) \left(-\Lambda(\tilde{U})\right)^k \tilde{\Upsilon}^{(k)} \right]_0 = &\mathcal{O}(\varepsilon)\mathcal{O}(|\tilde{\Upsilon}^{(k)}(t,0^\pm|))\\
        \nonumber& + \mathcal{O}\left(\sum_{k^\prime = 1}^k|\partial_t^{k^\prime}\eta(t)|\right)\\
        \nonumber& + \mathcal{O}\left(\sum_{k^\prime = 0}^{k-1} \Upsilon^{(k^\prime)}(t,0^\pm)\right).
    \end{align}
    In \eqref{equ:interior_loc_dis} and \eqref{equ:jump_loc_dis}, the $\mathcal{O}$-terms are justified by the smallness bounds \eqref{est:smallness_shift_loc_dis}, \eqref{est:smallness_perturbation_loc_dis} and the initial bound \eqref{est:boundedness_initial_data_loc} in combination
    with Sobolev's inequality and the $L^\infty$-bounds stated in Theorem \ref{thm:damping_estimate_nonlocalized_dis} for $K-1 \geq 1$.

    From \eqref{equ:jump_loc_dis} one finds successively that for $k \in \lbrace 1, \ldots, K\rbrace$
    \begin{align}
        \label{equ:eta_derivatives}
        \partial_t^k \eta(t) = \mathcal{O}\left(\sum_{k^\prime = 0}^{k-1} \Upsilon^{(k^\prime)}(t,0^\pm) \right).
    \end{align}
    Multiplying \eqref{equ:jump_loc_dis} by $\tilde{B}(t)$ with $\tilde{B}$ as in \eqref{def:tilde_B} and using \eqref{equ:eta_derivatives} yields
    \begin{align*}
        |\tilde{\Upsilon}^{(k),out}(t,0^+)| \leq C_L \left(|\tilde{\Upsilon}^{(k),in}(t,0^+)| + \sum_{k^\prime = 0}^{k-1} |\Upsilon^{(k^\prime)}(t,0^\pm)| \right)
    \end{align*}
    for $k \in \lbrace 0, \ldots, K\rbrace$ and some constant $C_L> 0$ that we have named here for later reference. The incoming and outgoing waves $\tilde{\Upsilon}^{(k),in}(t,x) \in \RR^{N+1}$ and $\tilde{\Upsilon}^{(k),out}(t,x) \in \RR^{N-1}$ are defined 
    analogously to \eqref{def:incoming_waves} and \eqref{def:outgoing_waves} for $(t,x) \in [0,T] \times (0,\infty)$. Applying Sobolev's embedding theorem we obtain
    \begin{align}
        \label{est:Kreiss_type_all}
        |\tilde{\Upsilon}^{(k),out}(t,0^+)| \leq & C_L |\tilde{\Upsilon}^{(k),in}(t,0^+)| + \delta  \norm{\tilde{\Upsilon}^{(k)}(t)}{L^2(\dot{\RR})}\\
        \nonumber & + C(\delta)\sum_{k^\prime = 0}^{k-1} \norm{\Upsilon^{(k^\prime)}(t)}{L^2(\dot{\RR})} 
    \end{align}
    for $\delta > 0$ and a constant $C(\delta) > 0$.

    Let $\Omega=\Omega(x)$ be the diagonal matrix containing the weights $\alpha_j$ as in \eqref{def:weighting_matrix} defined for $x \in \dot{\RR}$ with appropriate one-sided limits $\Omega(0^\pm)$ at $x=0$. To make use of the
    Kreiss-type estimate \eqref{est:Kreiss_type_all} an additional weight favoring incoming waves has to be introduced (cf.~\cite{YZ20}). To this end, define for $z>1$ the matrices
    \begin{align*}
        Z^- := \begin{pmatrix}
            I_{p-1} & 0_{(p-1) \times (N-p+1)}\\
            0_{(N-p+1) \times (p-1)} & z I_{N-p+1}
        \end{pmatrix}, \quad Z^+ := \begin{pmatrix}
            z I_{p} & 0_{p \times (N-p)}\\
            0_{(N-p) \times p} & I_{N-p}
        \end{pmatrix}.
    \end{align*}
    The diagonal matrix of weights
    \begin{align}
        \label{def:weighting_matrix_dis}
        \tilde{\Omega}(x) := \Omega(x) Z^\pm, \quad x \gtrless 0
    \end{align}
    defines an equivalent $L^2$-norm on $\dot{\RR}$ given by
    \begin{align*}
        \norm{f}{L^2_{\tilde{\Omega}}(\dot{\RR})} := \sqrt{\skp{f, \tilde{\Omega} f}{L^2(\dot{\RR})}}.
    \end{align*}
    Taking the time derivative of the weighted $L^2$ inner product 
    \[\skp{\tilde{\Upsilon}^{(k)}(t), \tilde{\Omega} \tilde{\Upsilon}^{(k)}(t)}{L^2(\dot{\RR})}\] 
    and partial integration yields the estimate
    \begin{align}
        \label{est:weighted_energy_dis}
        \frac{1}{2}\frac{d}{dt} \norm{\tilde{\Upsilon}^{(k)}(t)}{L^2_{\tilde{\Omega}}(\dot{\RR})}^2 \leq & -\frac{3}{4} \theta \norm{\tilde{\Upsilon}^{(k)}(t)}{L^2_{\tilde{\Omega}}(\dot{\RR})}^2 + \left[(\tilde{\Upsilon}^{(k)}(t))^T \tilde{\Omega} \Lambda \tilde{\Upsilon}^{(k)}(t)\right]_0\\
        \nonumber & + C \left(|\dot{\eta}(t)|^2 + \sum_{k^\prime = 0}^{k-1} \norm{\Upsilon^{(k^\prime)}(t)}{L^2(\dot{\RR})}^2\right)
    \end{align}
    for some positive constants $C, c >0$, small $\varepsilon$, and an appropriate choice for the constants $C_\alpha, c_\alpha$ defining the weights $\alpha_j$ in \eqref{definition_alpha_j} to eliminate terms with exponentially decaying factors. The trace terms on the right-hand side of \eqref{est:weighted_energy_dis}, written as a jump, can be estimated as
    \begin{align*}
        \left[(\tilde{\Upsilon}^{(k)}(t))^T \tilde{\Omega} \Lambda \tilde{\Upsilon}^{(k)}(t)\right]_0 \leq & -z \cdot c_{nc} \cdot \min_j \lbrace \alpha_j(0^\pm)\rbrace \left|\tilde{\Upsilon}^{(k),in}(t,0^+)\right|^2\\
        & + \max_j |\lambda_j(U^\pm)| \left|\tilde{\Upsilon}^{(k),out}(t,0^+)\right|^2
    \end{align*}
    Since outgoing waves are controlled by incoming waves \eqref{est:Kreiss_type_all} we find
    \begin{align}
        \label{est:boundary_dis}
        \Big[(\tilde{\Upsilon}^{(k)}(t))^T \tilde{\Omega} & \Lambda \tilde{\Upsilon}^{(k)}(t)\Big]_0 \leq\\
        \nonumber & \left(C_L\max_j |\lambda_j(U^\pm)| -z c_{nc} \min_j \lbrace \alpha_j(0^\pm)\rbrace\right) \left|\tilde{\Upsilon}^{(k),in}(t,0^+)\right|^2\\
        \nonumber & + \frac{\delta \max_j |\lambda_j(U^\pm)|}{\min_{j,x} \lbrace \alpha_j(x)\rbrace} \norm{\tilde{\Upsilon}^{(k)}(t)}{L^2_{\tilde{\Omega}}(\dot{\RR})}^2\\
        \nonumber & + C(\delta)\sum_{k^\prime = 0}^{k-1} \norm{\Upsilon^{(k^\prime)}(t)}{L^2(\dot{\RR})}^2.
    \end{align}
    With $z$ large enough and $\delta$ small enough, inserting the estimate \eqref{est:boundary_dis} into the energy estimate \eqref{est:weighted_energy_dis} implies
    \begin{align}
        \label{est:weighted_energy_dis_final}
        \frac{d}{dt} \norm{\tilde{\Upsilon}^{(k)}(t)}{L^2_{\tilde{\Omega}}(\dot{\RR})}^2 \leq -\theta \norm{\tilde{\Upsilon}^{(k)}(t)}{L^2_{\tilde{\Omega}}(\dot{\RR})}^2 + 2C \left(|\dot{\eta}|^2 + \sum_{k^\prime = 0}^{k-1} \norm{\Upsilon^{(k^\prime)}}{L^2(\dot{\RR})}^2\right)
    \end{align}
    for some large constant $C>0$. The forcing terms on the right-hand side in estimate \eqref{est:weighted_energy_dis_final} other than $|\dot{\eta}|^2$ and $\Upsilon^{(0)}$ are successively eliminated as demonstarted in the proof of Theorem \ref{thm:damping_estimate_nonlocalized_smooth}, which results in
    the damping estimates
    \begin{align*}
        \norm{\Upsilon^{(k)}(t)}{L^2(\dot{\RR})}^2 \leq & C e^{-\frac{\theta}{2^{k-1}}t} \norm{( \Upsilon^{(k^\prime)}(0) )_{1\leq k^\prime \leq k}}{L^2(\dot{\RR})}^2\\
        &  + C \int_0^t e^{-\frac{\theta}{2^{k-1}}(t-s)} \left(|\dot{\eta}(s)|^2 + \norm{\Upsilon^{(0)}(s)}{L^2(\dot{\RR})}^2\right) ds
    \end{align*}
    for $k = 1, \ldots, K$ and a large constant $C>0$. The claim of Theorem \ref{thm:damping_estimate_localized_dis} follows from the definitions of the variables $\Upsilon^{(k)}$ and the boundedness of $R(\tilde{U})$ and $L(\tilde{U})$.
\end{proof}
\textbf{Acknowledgments:} We want to thank K.~Zumbrun for proposing this problem and making helpful suggestions on how to
approach it.
\bibliographystyle{plain}
\bibliography{MOC_damping_estimates}

@Article{MZ05,
  author     = {Mascia, C. and Zumbrun, K.},
  title      = {Stability of large-amplitude shock profiles of general relaxation systems},
  journal    = {SIAM J. Math. Anal.},
  year       = {2005},
  volume     = {37},
  number     = {3},
  pages      = {889--913},
  issn       = {0036-1410,1095-7154},
  doi        = {10.1137/S0036141004435844},
  fjournal   = {SIAM Journal on Mathematical Analysis},
  mrclass    = {35L60 (35B35 35F20 76L05)},
  mrnumber   = {2191781},
  mrreviewer = {Thomas\ Hagen},
  url        = {https://doi.org/10.1137/S0036141004435844},
}

@Book{K95,
  title     = {Perturbation theory for linear operators},
  publisher = {Springer-Verlag, Berlin},
  year      = {1995},
  author    = {Kato, T.},
  series    = {Classics in Mathematics},
  isbn      = {3-540-58661-X},
  note      = {Reprint of the 1980 edition},
  mrclass   = {47A55 (46-00 47-00)},
  mrnumber  = {1335452},
  pages     = {xxii+619},
}

@Article{YZ20,
  author   = {Yang, Z. and Zumbrun, K.},
  title    = {Stability of hydraulic shock profiles},
  journal  = {Arch. Ration. Mech. Anal.},
  year     = {2020},
  volume   = {235},
  number   = {1},
  pages    = {195--285},
  issn     = {0003-9527,1432-0673},
  doi      = {10.1007/s00205-019-01422-4},
  fjournal = {Archive for Rational Mechanics and Analysis},
  mrclass  = {76L05 (35L60 76E30)},
  mrnumber = {4062477},
  url      = {https://doi.org/10.1007/s00205-019-01422-4},
}

@Article{SK85,
  author     = {Shizuta, Y. and Kawashima, S.},
  title      = {Systems of equations of hyperbolic-parabolic type with applications to the discrete {B}oltzmann equation},
  journal    = {Hokkaido Math. J.},
  year       = {1985},
  volume     = {14},
  number     = {2},
  pages      = {249--275},
  issn       = {0385-4035},
  doi        = {10.14492/hokmj/1381757663},
  fjournal   = {Hokkaido Mathematical Journal},
  mrclass    = {35M05 (35B40 76P05 82A40)},
  mrnumber   = {798756},
  mrreviewer = {Carlo\ Cercignani},
  url        = {https://doi.org/10.14492/hokmj/1381757663},
}

@Article{BW24,
  author     = {Blochas, P. and Wheeler, A.},
  title      = {Majda and {ZND} models for detonation: nonlinear stability vs. formation of singularities},
  journal    = {SIAM J. Math. Anal.},
  year       = {2024},
  volume     = {56},
  number     = {5},
  pages      = {6137--6191},
  issn       = {0036-1410,1095-7154},
  doi        = {10.1137/23M1544945},
  fjournal   = {SIAM Journal on Mathematical Analysis},
  mrclass    = {35L40 (35B40 35B44 35L67 80A25)},
  mrnumber   = {4793817},
  mrreviewer = {Antonio\ Bove},
  url        = {https://doi.org/10.1137/23M1544945},
}

@Article{MZ02,
  author     = {Mascia, C. and Zumbrun, K.},
  title      = {Pointwise {G}reen's function bounds and stability of relaxation shocks},
  journal    = {Indiana Univ. Math. J.},
  year       = {2002},
  volume     = {51},
  number     = {4},
  pages      = {773--904},
  issn       = {0022-2518,1943-5258},
  doi        = {10.1512/iumj.2002.51.2212},
  fjournal   = {Indiana University Mathematics Journal},
  mrclass    = {35L60 (35A08 35B45 47N20 76L05)},
  mrnumber   = {1947862},
  mrreviewer = {Vladimir\ Tulovsky},
  url        = {https://doi.org/10.1512/iumj.2002.51.2212},
}

@Article{S25,
  author   = {Sroczinski, M.},
  title    = {Global existence and asymptotic decay for small solutions of general quasilinear hyperbolic balance laws},
  journal  = {J. Hyperbolic Differ. Equ.},
  year     = {2025},
  volume   = {22},
  number   = {4},
  pages    = {613--642},
  issn     = {0219-8916,1793-6993},
  doi      = {10.1142/S0219891625500171},
  fjournal = {Journal of Hyperbolic Differential Equations},
  mrclass  = {35L51 (35A01 35B35)},
  mrnumber = {5010756},
  url      = {https://doi.org/10.1142/S0219891625500171},
}

@Article{K83,
  author  = {Kawashima, S.},
  title   = {Systems of a hyperbolic-parabolic composite type, with applications to the equations of magnetohydrodynamics,},
  journal = {Ph. D. thesis, Kyoto University},
  year    = {1983},
}

@Article{KY04,
  author     = {Kawashima, S. and Yong, W.-A.},
  title      = {Dissipative structure and entropy for hyperbolic systems of balance laws},
  journal    = {Arch. Ration. Mech. Anal.},
  year       = {2004},
  volume     = {174},
  number     = {3},
  pages      = {345--364},
  issn       = {0003-9527,1432-0673},
  doi        = {10.1007/s00205-004-0330-9},
  fjournal   = {Archive for Rational Mechanics and Analysis},
  mrclass    = {35L60 (35L65)},
  mrnumber   = {2107774},
  mrreviewer = {Sylvie\ Benzoni-Gavage},
  url        = {https://doi.org/10.1007/s00205-004-0330-9},
}

@Article{BHN07,
  author     = {Bianchini, S. and Hanouzet, B. and Natalini, R.},
  title      = {Asymptotic behavior of smooth solutions for partially dissipative hyperbolic systems with a convex entropy},
  journal    = {Comm. Pure Appl. Math.},
  year       = {2007},
  volume     = {60},
  number     = {11},
  pages      = {1559--1622},
  issn       = {0010-3640,1097-0312},
  doi        = {10.1002/cpa.20195},
  fjournal   = {Communications on Pure and Applied Mathematics},
  mrclass    = {35L60 (35B40 35L45)},
  mrnumber   = {2349349},
  mrreviewer = {Fumioki\ Asakura},
  url        = {https://doi.org/10.1002/cpa.20195},
}

@Article{BZ11,
  author     = {Beauchard, K. and Zuazua, E.},
  title      = {Large time asymptotics for partially dissipative hyperbolic systems},
  journal    = {Arch. Ration. Mech. Anal.},
  year       = {2011},
  volume     = {199},
  number     = {1},
  pages      = {177--227},
  issn       = {0003-9527,1432-0673},
  doi        = {10.1007/s00205-010-0321-y},
  fjournal   = {Archive for Rational Mechanics and Analysis},
  mrclass    = {35L60 (35B35 35B40 35C07 35L40)},
  mrnumber   = {2754341},
  mrreviewer = {Toan\ Nguyen},
  url        = {https://doi.org/10.1007/s00205-010-0321-y},
}

@Article{Zeng_99,
  author     = {Zeng, Y.},
  title      = {Gas dynamics in thermal nonequilibrium and general hyperbolic systems with relaxation},
  journal    = {Arch. Ration. Mech. Anal.},
  year       = {1999},
  volume     = {150},
  number     = {3},
  pages      = {225--279},
  issn       = {0003-9527,1432-0673},
  doi        = {10.1007/s002050050188},
  fjournal   = {Archive for Rational Mechanics and Analysis},
  mrclass    = {35L40 (76N10 76N15 76V05)},
  mrnumber   = {1738119},
  mrreviewer = {Christophe\ Cheverry},
  url        = {https://doi.org/10.1007/s002050050188},
}

@Article{M83,
  author     = {Majda, A.},
  title      = {The stability of multidimensional shock fronts},
  journal    = {Mem. Amer. Math. Soc.},
  year       = {1983},
  volume     = {41},
  number     = {275},
  pages      = {iv+95},
  issn       = {0065-9266,1947-6221},
  doi        = {10.1090/memo/0275},
  fjournal   = {Memoirs of the American Mathematical Society},
  mrclass    = {35L67 (76L05)},
  mrnumber   = {683422},
  mrreviewer = {Ronald\ DiPerna},
  url        = {https://doi.org/10.1090/memo/0275},
}

@Book{D16,
  title      = {Hyperbolic conservation laws in continuum physics},
  publisher  = {Springer-Verlag, Berlin},
  year       = {2016},
  author     = {Dafermos, C.~M.},
  volume     = {325},
  series     = {Grundlehren der mathematischen Wissenschaften},
  edition    = {4th},
  isbn       = {978-3-662-49449-3; 978-3-662-49451-6},
  doi        = {10.1007/978-3-662-49451-6},
  mrclass    = {35L65 (35-02 35L67 74J40 76-02)},
  mrnumber   = {3468916},
  mrreviewer = {Marta\ Lewicka},
  pages      = {xxxviii+826},
  url        = {https://doi.org/10.1007/978-3-662-49451-6},
}

@Book{BGD07,
  title      = {Multidimensional hyperbolic partial differential equations},
  publisher  = {The Clarendon Press, Oxford University Press, Oxford},
  year       = {2007},
  author     = {Benzoni-Gavage, S. and Serre, D.},
  series     = {Oxford Mathematical Monographs},
  isbn       = {978-0-19-921123-4; 0-19-921123-X},
  note       = {First-order systems and applications},
  mrclass    = {35-02 (35L50 35L65 35L67 76L05 76N10 76N15)},
  mrnumber   = {2284507},
  mrreviewer = {Kenneth\ H.\ Karlsen},
  pages      = {xxvi+508},
}

@InCollection{M01,
  author     = {M\'etivier, G.},
  title      = {Stability of multidimensional shocks},
  booktitle  = {Advances in the theory of shock waves},
  publisher  = {Birkh\"auser Boston, Boston, MA},
  year       = {2001},
  volume     = {47},
  series     = {Progr. Nonlinear Differential Equations Appl.},
  pages      = {25--103},
  isbn       = {0-8176-4187-4},
  mrclass    = {35L67 (35B35 35L50 35L60 35L65 35S50)},
  mrnumber   = {1842775},
  mrreviewer = {Fumioki\ Asakura},
}

@Article{K70,
  author     = {Kreiss, H.-O.},
  title      = {Initial boundary value problems for hyperbolic systems},
  journal    = {Comm. Pure Appl. Math.},
  year       = {1970},
  volume     = {23},
  pages      = {277--298},
  issn       = {0010-3640,1097-0312},
  doi        = {10.1002/cpa.3160230304},
  fjournal   = {Communications on Pure and Applied Mathematics},
  mrclass    = {35L50},
  mrnumber   = {437941},
  mrreviewer = {Ja.\ Zitomirskii},
  url        = {https://doi.org/10.1002/cpa.3160230304},
}

@InCollection{G91,
  author    = {Goodman, J.},
  title     = {Remarks on the stability of viscous shock waves},
  booktitle = {Viscous profiles and numerical methods for shock waves ({R}aleigh, {NC}, 1990)},
  publisher = {SIAM, Philadelphia, PA},
  year      = {1991},
  pages     = {66--72},
  isbn      = {0-89871-283-1},
  mrclass   = {35L67},
  mrnumber  = {1142641},
}

@Article{MZ04_2,
  author     = {Mascia, C. and Zumbrun, K.},
  title      = {Stability of large-amplitude viscous shock profiles of hyperbolic-parabolic systems},
  journal    = {Arch. Ration. Mech. Anal.},
  year       = {2004},
  volume     = {172},
  number     = {1},
  pages      = {93--131},
  issn       = {0003-9527,1432-0673},
  doi        = {10.1007/s00205-003-0293-2},
  fjournal   = {Archive for Rational Mechanics and Analysis},
  mrclass    = {35L65 (35B25 35B30 35B40 35K65 35L67)},
  mrnumber   = {2048568},
  mrreviewer = {J\"org\ M.\ H\"arterich},
  url        = {https://doi.org/10.1007/s00205-003-0293-2},
}

@Article{MZ04_1,
  author     = {Mascia, C. and Zumbrun, K.},
  title      = {Stability of small-amplitude shock profiles of symmetric hyperbolic-parabolic systems},
  journal    = {Comm. Pure Appl. Math.},
  year       = {2004},
  volume     = {57},
  number     = {7},
  pages      = {841--876},
  issn       = {0010-3640,1097-0312},
  doi        = {10.1002/cpa.20023},
  fjournal   = {Communications on Pure and Applied Mathematics},
  mrclass    = {35B30 (35K55 35L65 35Q35 76L05 76N10 76W05)},
  mrnumber   = {2044067},
  mrreviewer = {Fr\'ed\'eric\ Rousset},
  url        = {https://doi.org/10.1002/cpa.20023},
}

@Article{JNRZ_14,
  author     = {Johnson, M.~A. and Noble, P. and Rodrigues, L.~M. and Zumbrun, K.},
  title      = {Behavior of periodic solutions of viscous conservation laws under localized and nonlocalized perturbations},
  journal    = {Invent. Math.},
  year       = {2014},
  volume     = {197},
  number     = {1},
  pages      = {115--213},
  issn       = {0020-9910,1432-1297},
  doi        = {10.1007/s00222-013-0481-0},
  fjournal   = {Inventiones Mathematicae},
  mrclass    = {35L65 (35B10 35B35 35B40 35C07 35K59 35L60)},
  mrnumber   = {3219516},
  mrreviewer = {Ramon\ G.\ Plaza},
  url        = {https://doi.org/10.1007/s00222-013-0481-0},
}

@Article{JZN_11,
  author   = {Johnson, M.~A. and Zumbrun, K. and Noble, P.},
  title    = {Nonlinear stability of viscous roll waves},
  journal  = {SIAM J. Math. Anal.},
  year     = {2011},
  volume   = {43},
  number   = {2},
  pages    = {577--611},
  issn     = {0036-1410,1095-7154},
  doi      = {10.1137/100785454},
  fjournal = {SIAM Journal on Mathematical Analysis},
  mrclass  = {35Q35 (35B35 76D33)},
  mrnumber = {2784868},
  url      = {https://doi.org/10.1137/100785454},
}

@InCollection{Z_04_1,
  author     = {Zumbrun, K.},
  title      = {Stability of large-amplitude shock waves of compressible {N}avier-{S}tokes equations},
  booktitle  = {Handbook of mathematical fluid dynamics. {V}ol. {III}},
  publisher  = {North-Holland, Amsterdam},
  year       = {2004},
  pages      = {311--533},
  isbn       = {0-444-51556-9},
  note       = {With an appendix by Helge Kristian Jenssen and Gregory Lyng},
  mrclass    = {35Q35 (35-02 35B35 35L67 35Q30 76L05 76N10)},
  mrnumber   = {2099037},
  mrreviewer = {Rapha\"el\ Danchin},
}

@Article{HoZ95,
  author     = {Hoff, D. and Zumbrun, K.},
  title      = {Multi-dimensional diffusion waves for the {N}avier-{S}tokes equations of compressible flow},
  journal    = {Indiana Univ. Math. J.},
  year       = {1995},
  volume     = {44},
  number     = {2},
  pages      = {603--676},
  issn       = {0022-2518,1943-5258},
  doi        = {10.1512/iumj.1995.44.2003},
  fjournal   = {Indiana University Mathematics Journal},
  mrclass    = {35Q30 (76N10)},
  mrnumber   = {1355414},
  mrreviewer = {\v S\'arka\ Matu\v s\ocirc u-Ne\v casov\'a},
  url        = {https://doi.org/10.1512/iumj.1995.44.2003},
}

@Article{Yong_Zumbrun_00,
  author     = {Yong, W.-A. and Zumbrun, K.},
  title      = {Existence of relaxation shock profiles for hyperbolic conservation laws},
  journal    = {SIAM J. Appl. Math.},
  year       = {2000},
  volume     = {60},
  number     = {5},
  pages      = {1565--1575},
  issn       = {0036-1399,1095-712X},
  doi        = {10.1137/S0036139999352705},
  fjournal   = {SIAM Journal on Applied Mathematics},
  mrclass    = {35L65 (34C99 35L67)},
  mrnumber   = {1761762},
  mrreviewer = {Yue-Jun\ Peng},
  url        = {https://doi.org/10.1137/S0036139999352705},
}

@Misc{Zumbrun_26,
  author        = {K. Zumbrun},
  title         = {Pseudodifferential damping estimates and stability of relaxation shocks},
  year          = {2026},
  archiveprefix = {arXiv},
  eprint        = {2407.14484},
  primaryclass  = {math.AP},
  url           = {https://arxiv.org/abs/2407.14484},
}

@Article{Broadwell_64,
  author   = {Broadwell, J.~E.},
  title    = {Shock Structure in a Simple Discrete Velocity Gas},
  journal  = {The Physics of Fluids},
  year     = {1964},
  volume   = {7},
  number   = {8},
  pages    = {1243-1247},
  month    = {08},
  issn     = {0031-9171},
  doi      = {10.1063/1.1711368},
  eprint   = {https://pubs.aip.org/aip/pfl/article-pdf/7/8/1243/12408805/1243_1_online.pdf},
  url      = {https://doi.org/10.1063/1.1711368},
}

@Unpublished{Rodrigues_Yang_Zumbrun_25,
  author      = {Rodrigues, L.~M. and Yang, Z. and Zumbrun, K.},
  title       = {{Existence and Stability of Hydraulic Shock Profiles of the Richard-Gavrilyuk Model for Inclined Shallow Water Flow}},
  note        = {This is work in progress. Shared for grant application purposes.},
  month       = Nov,
  year        = {2025},
  hal_id      = {hal-05343041},
  hal_version = {v1},
  url         = {https://hal.science/hal-05343041},
}

@Article{Aregba-Driolle_Hanouzet_11,
  author     = {Aregba-Driollet, D. and Hanouzet, B.},
  title      = {Kerr-{D}ebye relaxation shock profiles for {K}err equations},
  journal    = {Commun. Math. Sci.},
  year       = {2011},
  volume     = {9},
  number     = {1},
  pages      = {1--31},
  issn       = {1539-6746,1945-0796},
  doi        = {10.4310/cms.2011.v9.n1.a1},
  fjournal   = {Communications in Mathematical Sciences},
  mrclass    = {35L67 (35L65 35Q60 78A60)},
  mrnumber   = {2836834},
  mrreviewer = {Simon\ Labrunie},
  url        = {https://doi.org/10.4310/cms.2011.v9.n1.a1},
}

@Article{Dressel_Yong_06,
  author     = {Dressel, Alexander and Yong, Wen-An},
  title      = {Existence of traveling-wave solutions for hyperbolic systems of balance laws},
  journal    = {Arch. Ration. Mech. Anal.},
  year       = {2006},
  volume     = {182},
  number     = {1},
  pages      = {49--75},
  issn       = {0003-9527,1432-0673},
  doi        = {10.1007/s00205-006-0430-9},
  fjournal   = {Archive for Rational Mechanics and Analysis},
  mrclass    = {35L60},
  mrnumber   = {2247952},
  mrreviewer = {Benjamin\ Texier},
  url        = {https://doi.org/10.1007/s00205-006-0430-9},
}

@Article{Wang_13,
  author     = {Xie, W.},
  title      = {Shock profile for gas dynamics in thermal nonequilibrium},
  journal    = {Involve},
  year       = {2013},
  volume     = {6},
  number     = {3},
  pages      = {311--322},
  issn       = {1944-4176,1944-4184},
  doi        = {10.2140/involve.2013.6.311},
  fjournal   = {Involve. A Journal of Mathematics},
  mrclass    = {35L65 (35L67)},
  mrnumber   = {3101763},
  mrreviewer = {Ramon\ G.\ Plaza},
  url        = {https://doi.org/10.2140/involve.2013.6.311},
}

@InCollection{Yong_01,
  author     = {Yong, W.-A.},
  title      = {Basic aspects of hyperbolic relaxation systems},
  booktitle  = {Advances in the theory of shock waves},
  publisher  = {Birkh\"auser Boston, Boston, MA},
  year       = {2001},
  volume     = {47},
  series     = {Progr. Nonlinear Differential Equations Appl.},
  pages      = {259--305},
  isbn       = {0-8176-4187-4},
  mrclass    = {35L60 (35B25 35L65)},
  mrnumber   = {1842777},
  mrreviewer = {Roberto\ Natalini},
}

@Article{Liu_87,
  author     = {Liu, T.-P.},
  title      = {Hyperbolic conservation laws with relaxation},
  journal    = {Comm. Math. Phys.},
  year       = {1987},
  volume     = {108},
  number     = {1},
  pages      = {153--175},
  issn       = {0010-3616,1432-0916},
  fjournal   = {Communications in Mathematical Physics},
  mrclass    = {35L65},
  mrnumber   = {872145},
  mrreviewer = {Alexander\ Doktor},
  url        = {http://projecteuclid.org/euclid.cmp/1104116362},
}

@Article{Richard_Gavrilyuk_12,
  author  = {Richard, G. L. and Gavrilyuk, S. L.},
  title   = {A new model of roll waves: comparison with Brock’s experiments},
  journal = {Journal of Fluid Mechanics},
  year    = {2012},
  volume  = {698},
  pages   = {374–405},
  doi     = {10.1017/jfm.2012.96},
}

@Article{Richard_Gavrilyuk_13,
  author  = {Richard, G. L. and Gavrilyuk, S. L.},
  title   = {The classical hydraulic jump in a model of shear shallow-water flows},
  journal = {Journal of Fluid Mechanics},
  year    = {2013},
  volume  = {725},
  pages   = {492–521},
  doi     = {10.1017/jfm.2013.174},
}

@Article{Boillat_Ruggeri_98,
  author     = {Boillat, G. and Ruggeri, T.},
  title      = {On the shock structure problem for hyperbolic system of balance laws and convex entropy},
  journal    = {Contin. Mech. Thermodyn.},
  year       = {1998},
  volume     = {10},
  number     = {5},
  pages      = {285--292},
  issn       = {0935-1175,1432-0959},
  doi        = {10.1007/s001610050094},
  fjournal   = {Continuum Mechanics and Thermodynamics},
  mrclass    = {76L05 (35L67 35Q99 80A99)},
  mrnumber   = {1652858},
  mrreviewer = {Philippe\ G.\ LeFloch},
  url        = {https://doi.org/10.1007/s001610050094},
}

@Book{Whitham_74,
  title     = {Linear and nonlinear waves},
  publisher = {Wiley-Interscience [John Wiley \& Sons], New York-London-Sydney},
  year      = {1974},
  author    = {Whitham, G. B.},
  series    = {Pure and Applied Mathematics},
  mrclass   = {76.00 (35L05 35Q99)},
  mrnumber  = {483954},
  pages     = {xvi+636},
}

@Article{Alexopoulos_deRijk_26,
  author   = {Alexopoulos, J. and de Rijk, B.},
  title    = {Nonlinear stability of periodic wave trains in the {F}itz{H}ugh-{N}agumo system against fully nonlocalized perturbations},
  journal  = {J. Differential Equations},
  year     = {2026},
  volume   = {457},
  pages    = {Paper No. 114013, 65},
  issn     = {0022-0396,1090-2732},
  doi      = {10.1016/j.jde.2025.114013},
  fjournal = {Journal of Differential Equations},
  mrclass  = {35B35 (35B10 35K57 44A10)},
  mrnumber = {4998758},
  url      = {https://doi.org/10.1016/j.jde.2025.114013},
}

@Article{Li_Liu_09,
  author   = {Li, T. and Liu, H.},
  title    = {Critical thresholds in hyperbolic relaxation systems},
  journal  = {J. Differential Equations},
  year     = {2009},
  volume   = {247},
  number   = {1},
  pages    = {33--48},
  issn     = {0022-0396,1090-2732},
  doi      = {10.1016/j.jde.2009.03.032},
  fjournal = {Journal of Differential Equations},
  mrclass  = {35L60},
  mrnumber = {2510127},
  url      = {https://doi.org/10.1016/j.jde.2009.03.032},
}

@Misc{Bressan_Caravenna_Shen_26,
  author        = {A. Bressan and L. Caravenna and W. Shen},
  title         = {Local Asymptotic Patterns for Viscous Approximations of Conservation Laws},
  year          = {2026},
  archiveprefix = {arXiv},
  eprint        = {2605.00189},
  primaryclass  = {math.AP},
  url           = {https://arxiv.org/abs/2605.00189},
}

@Article{Rodrigues_Zumbrun_16,
  author     = {Rodrigues, L.~M. and Zumbrun, K.},
  title      = {Periodic-coefficient damping estimates, and stability of large-amplitude roll waves in inclined thin film flow},
  journal    = {SIAM J. Math. Anal.},
  year       = {2016},
  volume     = {48},
  number     = {1},
  pages      = {268--280},
  issn       = {0036-1410,1095-7154},
  doi        = {10.1137/15M1016242},
  fjournal   = {SIAM Journal on Mathematical Analysis},
  mrclass    = {35Q35 (35B10 35B35 35K59)},
  mrnumber   = {3449905},
  mrreviewer = {Roman\ M.\ Taranets},
  url        = {https://doi.org/10.1137/15M1016242},
}

@Misc{Alexopoulos_Baerlin_Yang_Zumbrun_26,
  author       = {J. Alexopoulos and J. B{\"a}rlin and Z. Yang and K. Zumbrun},
  title        = {Stability of discontinuous shock profiles of general relaxation systems},
  howpublished = {In preparation},
  year         = {2026},
}

@Article{Hanouzet_Natalini_03,
  author     = {Hanouzet, B. and Natalini, R.},
  title      = {Global existence of smooth solutions for partially dissipative hyperbolic systems with a convex entropy},
  journal    = {Arch. Ration. Mech. Anal.},
  year       = {2003},
  volume     = {169},
  number     = {2},
  pages      = {89--117},
  issn       = {0003-9527,1432-0673},
  doi        = {10.1007/s00205-003-0257-6},
  fjournal   = {Archive for Rational Mechanics and Analysis},
  mrclass    = {35L60 (35B40 35L45)},
  mrnumber   = {2005637},
  mrreviewer = {L.\ Hsiao},
  url        = {https://doi.org/10.1007/s00205-003-0257-6},
}

@Book{Li_Yu_85,
  title     = {Boundary value problems for quasilinear hyperbolic systems},
  publisher = {Duke University, Mathematics Department, Durham, NC},
  year      = {1985},
  author    = {Li, T. and Yu, W.},
  volume    = {V},
  series    = {Duke University Mathematics Series},
  mrclass   = {35Lxx},
  mrnumber  = {823237},
  pages     = {viii+325},
}

@Book{Taylor_91,
  title      = {Pseudodifferential operators and nonlinear {PDE}},
  publisher  = {Birkh\"auser Boston, Inc., Boston, MA},
  year       = {1991},
  author     = {Taylor, M.~E.},
  volume     = {100},
  series     = {Progress in Mathematics},
  isbn       = {0-8176-3595-5},
  doi        = {10.1007/978-1-4612-0431-2},
  mrclass    = {35Sxx (35G20 46E35 46F99 47G30 58G15)},
  mrnumber   = {1121019},
  mrreviewer = {William\ Margulies},
  pages      = {213},
  url        = {https://doi.org/10.1007/978-1-4612-0431-2},
}

\end{document}